\documentclass[11pt,letterpaper]{article}

\usepackage[margin=1in]{geometry}
\usepackage[colorlinks=true,linkcolor=blue,citecolor=blue]{hyperref}
\usepackage{graphicx,subfigure,amsmath,amssymb,amsfonts,bm,epsfig,epsf,url,dsfont,color}
\usepackage{amsthm}
\usepackage{natbib}

\usepackage[utf8]{inputenc} 
\usepackage[T1]{fontenc}    
\usepackage{hyperref}       
\usepackage{url}           
\usepackage{booktabs} 
\usepackage{amsfonts} 
\usepackage{nicefrac}   
\usepackage{microtype}

\usepackage[lined,boxed,ruled,norelsize,algo2e,linesnumbered]{algorithm2e}
\usepackage{amsmath}
\usepackage{amssymb}

\renewcommand{\phi}{\varphi}

\newcommand{\R}{\mathbb{R}}

\newcommand{\cW}{\mathcal{W}}

\newcommand{\cN}{\mathcal{N}}

\def\ds1{\mathds{1}}
\renewcommand{\epsilon}{\varepsilon}

\newcommand{\argmax}{\mathop{\mathrm{arg\,max}}}
\newcommand{\argmin}{\mathop{\mathrm{arg\,min}}}

\renewcommand{\tilde}{\widetilde}

\newlength{\minipagewidth}
\newcommand{\beq}{\begin{equation}}
\newcommand{\eeq}{\end{equation}}

\newcommand{\beqa}{\begin{eqnarray}}
\newcommand{\eeqa}{\end{eqnarray}}

\newcommand{\beqan}{\begin{eqnarray*}}
\newcommand{\eeqan}{\end{eqnarray*}}

\def\ba#1\ea{\begin{align*}#1\end{align*}} 
\def\banum#1\eanum{\begin{align}#1\end{align}}

\newcommand{\norm}[1]{\|#1\|}

\newcommand{\oracle}{\mathcal{T}}
\newcommand{\oracleprox}{\oracle_{\mathrm{prox}}}
\newcommand{\oraclenoiseygrad}{\oracle_{\mathrm{grad}}}
\newcommand{\grad}{\nabla}
\newcommand{\defeq}{\stackrel{\mathrm{{\scriptscriptstyle def}}}{=}}
\newcommand{\normFull}[1]{\left\|#1\right\|}
\newcommand{\E}{\mathbb{E}}
\newcommand{\Prob}{\mathbb{P}}
\newcommand{\op}{\mathrm{op}}	

\usepackage{thm-restate}

\newtheorem{theo}{Theorem}
\newtheorem{theorem}[theo]{Theorem}
\newtheorem{lemma}[theo]{Lemma}
\newtheorem{lem}[theo]{Lemma}
\newtheorem{definition}[theo]{Definition}
\newtheorem{rem}[theo]{Remark}

\usepackage{thmtools}

\usepackage{color}

\title{Complexity of Highly Parallel \\
Non-Smooth Convex Optimization}

\def\And{\and}
\author{
S\'ebastien Bubeck\\
	Microsoft Research\\
	\texttt{sebubeck@microsoft.com}
	\And
	Qijia Jiang \\
	Stanford University \\
	\texttt{qjiang2@stanford.edu}	
	\And
	Yin Tat Lee \thanks{Research supported in part by NSF Awards CCF-1740551, CCF-1749609, and DMS-1839116.}\\
	University of Washington \\ \& Microsoft Research\\
	\texttt{yintat@uw.edu}
	\And
	Yuanzhi Li  \\
	Carnegie Mellon University\\
	\texttt{yuanzhil@andrew.cmu.edu}
	\And
	Aaron Sidford \thanks{Research supported in part by NSF CAREER Award CCF-1844855.}\\
	Stanford University \\
	\texttt{sidford@stanford.edu}
}

\begin{document}

\begin{titlepage}
	\clearpage\maketitle
	\thispagestyle{empty}

\begin{abstract}
A landmark result of non-smooth convex optimization is that gradient descent
 is an optimal algorithm whenever the number of computed gradients is smaller than the dimension $d$. In this paper we study the extension of this result to the parallel optimization setting. Namely we consider optimization algorithms interacting with a highly 
parallel gradient oracle, that is one that can answer $\mathrm{poly}(d)$ gradient queries in parallel. We show that in this case gradient descent is optimal only up to $\tilde{O}(\sqrt{d})$ rounds of interactions with the oracle. The lower bound improves upon a decades old construction by Nemirovski which proves optimality only up to $d^{1/3}$ rounds (as recently observed by Balkanski and Singer), and the suboptimality of gradient descent after $\sqrt{d}$ rounds was already observed by Duchi, Bartlett and Wainwright. In the latter regime we propose a new method with improved complexity, which we conjecture to be optimal. The analysis of this new method is based upon a generalized version of the recent results on optimal acceleration for highly smooth convex optimization.
\end{abstract}

\end{titlepage}

\newpage
\section{Introduction} \label{sec:intro}
Much of the research in convex optimization has focused on the {\em oracle model}, where an algorithm optimizing some objective function $f : \R^d \rightarrow \R$ does so by sequential interaction with, e.g., a gradient oracle
(given a query $x \in \R^d$, the oracle returns $\nabla f(x)$), \citep{NY83, Nes04, Bub15}.\footnote{Throughout we assume that $f$ is differentiable, though our results carry over to the case where $f$ is non-differentiable and given by a sub-gradient oracle. This generalization is immediate as our analysis and algorithms are stable under finite-precision arithmetic and convex functions are almost everywhere differentiable.
} In the early 1990s, 
Arkadi Nemirovski introduced the parallel version of this problem \citep{Nem94}: instead of submitting queries one by one sequentially, the algorithm can submit in parallel up to $Q \geq 1$ queries. We refer to the {\em depth} of such a parallel algorithm as the number of rounds of interaction with the oracle, and the {\em work} as the total number of queries (in particular work $\leq Q \times$ depth). In this paper we study the optimal depth achievable for highly parallel algorithms, namely we consider the regime $Q = \mathrm{poly}(d)$. We focus on non-smooth convex optimization, that is we want to optimize a Lipschitz, convex function $f$ on the unit Euclidean ball.

Our key result is a new form a quadratic acceleration: while for purely sequential methods the critical depth at which one can improve upon local search is $\tilde{O}(d)$, we show that in the highly parallel regime the critical depth is $\tilde{O}(\sqrt{d})$.  

\subsection{Classical optimality results}
Classically, when $Q=1$, it is known that gradient descent's query complexity is {\em order optimal} for any target accuracy $\epsilon$ in the range $\left[ d^{-1/2} , 1 \right]$.
More precisely, it is known that the query complexity of gradient descent is $O(1/\epsilon^2)$ and that for any $\epsilon$ in the range $\left[ d^{-1/2} , 1 \right]$, and for any algorithm, there exists a Lipschitz and convex function $f$ on which the number of oracle queries the algorithm makes to achieve additive $\epsilon$ accuracy is $\Omega(1/\epsilon^2)$.
Furthermore,  whenever $\epsilon$ is smaller than $d^{-1/2}$ there exists a better algorithm (i.e., with smaller depth), namely the center of gravity whose depth is $O(d \log(1/\epsilon))$. Consequently, an alternative formulation of these results is that, for $Q=1$, gradient descent is order optimal if and only if the depth is smaller than $\tilde{O}(d)$. (See previously cited references for the exact statements.)

\subsection{Optimality for highly parallel algorithms}
The main result of this paper
is to show that in the highly parallel regime ($Q= \mathrm{poly}(d)$), gradient descent is order optimal if and only if the depth is smaller than $\tilde{O}(\sqrt{d})$. Thus one has a ``quadratic" improvement over the purely sequential setting in terms of the critical depth at which naive local search becomes suboptimal.

The {\em only if} part of the above statement follows from \cite{DBW12}, where {\em randomized smoothing} with accelerated gradient descent was proposed (henceforth referred to as distributed randomized smoothing \citep{SBBML18}), and shown to achieve depth $d^{1/4} / \epsilon$, which is order better than $1/\epsilon^2$ exactly when the latter is equal to $\sqrt{d}$. 
A first key contribution of our work
is a matching lower bound showing that, when the depth is smaller than $\tilde{O}(\sqrt{d})$, no significant improvement over gradient descent is possible, i.e., $Q=1$ and $Q=\mathrm{poly}(d)$ have essentially the same power. Importantly we note that our lower bound applies to randomized algorithms. The previous state of the art lower bound was that gradient descent is optimal up to depth $\tilde{O}(d^{1/3})$ \citep{BS18}. In fact the construction in the latter paper is exactly the same as the original construction of Nemirovski in \citep{Nem94} (however the final statements are different, as Nemirovski was concerned with an $\ell_{\infty}$ setting instead of $\ell_2$, see also \cite{DG18} for more results about non-Euclidean setups).

A second key contribution of this work is to improve the state of the art complexity of parallel algorithms with depth between $\sqrt{d}$ and $d$. Improving the depth $d^{1/4} / \epsilon$ of \cite{DBW12} was explicitly mentioned as an open problem by \citet{SBBML18}. Leveraging the recent higher order acceleration schemes of \citet{Gasnikov18, JWZ18, BJLLS8}, we propose a new method with depth $d^{1/3} / \epsilon^{2/3}$. This means that for any value of $\epsilon$ in the range $\left[d^{-1}, d^{-1/2} \right]$ there is an algorithm that is order better than {\em both} gradient descent and center of gravity. Moreover we conjecture that the depth $d^{1/3} / \epsilon^{2/3}$ is in fact optimal for any $\epsilon$ in this range. We leave this question, as well as the question of the optimal work among optimal depth algorithms, for future works.

\subsection{Related works}
Though Nemirovski's prescient work stood alone for decades, more recently
the subfield of parallel/distributed optimization is booming,
propelled by problems in machine learning, see e.g., \citep{BPCPE11}. Chief among those problems is how to leverage mini-batches in stochastic gradient descent as efficiently as possible \citep{DGBSX12}. The literature on this topic is sprawling, see for example \citep{DRY18} which studies the total work achievable in parallel stochastic convex optimization, or \citep{ZX18} where the stochastic assumptions are leveraged to take advantage of second order information. More directly related to our work is \citep{Nem94, DG18, BS18} from the lower bound side (we directly improve upon the result in the latter paper), and \citep{DBW12, SBBML18} from the upper bound side (we directly improve upon the depth provided by the algorithms in those works). 

\section{Lower bound} \label{sec:lower}
Fix $\epsilon>0$ such that $1/\epsilon^2 = \tilde{O}(\sqrt{d})$. In this section we construct a random function $f$ such that, for any deterministic algorithm with depth $O(1/\epsilon^2)$ and total work $\mathrm{poly}(d)$, the output point $x$ is such that $\mathbb{E} [ f(x) - f^* ] > \epsilon$, where the expectation is with respect to the random function $f$, and $f^*$ denotes the minimum value of $f$ on the unit centered Euclidean ball. Note that by the minimax theorem, this implies that for any randomized algorithm there exists a deterministic function such that the same conclusion applies.  Formally, we prove the following: 

\begin{restatable}[Lower Bound]{theorem}{lowerbound}
	\label{thm:lower_bound}
	Let $\rho \in (0,1)$ and $C = 12 + 4 \log_d(Q/\rho)$. Further, assume that it holds that $\log(N) N \sqrt{C \log(d) / d} \leq \frac14$ (i.e., $N \lesssim \sqrt{d / \log^3(d)}$). Fix a randomized algorithm that queries at most $Q$ points on the unit ball per iteration (both function value and gradient), and that runs for at most $N$ iterations. Then, with probability at least $1-\rho$, when run on the shielded Nemirovski function $f$ (see Section~\ref{sec:Wall} and Section~\ref{sec:proofLB})) one has for any queried point:
	$f(x) - f^* \geq \frac{1}{4 \sqrt{N}}$.
\end{restatable}

The details of the proof of this theorem are deferred to Appendix~\ref{sec:LBapp}. In the remainder of this section we instead provide a sketch of its proof. We first recall in Section \ref{sec:NY} why, for purely sequential algorithms, the above statement holds true, and in fact one can even replace $\sqrt{d}$ by $d$ in this case (this construction goes back to \citep{YN76}, see also \citep{NY83}). Next, in Section \ref{sec:Nem94} we explain \citet{Nem94}'s construction, which yields a weaker version of the above statement, with $\sqrt{d}$ replaced by $d^{1/3}$ (as rediscovered by \citep{BS18}). We then explain in Section \ref{sec:Wall} our key new construction, a type of {\em shielding} operation. Finally, we conclude the proof sketch in Section \ref{sec:proofLB}. 

For the rest of the section we let $v_1, \ldots, v_N$ denote $N$ random orthonormal vectors in $\R^d$ (in particular $N \leq d$), and $x^* = - \frac{1}{\sqrt{N}} \sum_{i=1}^N v_i$. We define the {\em Nemirovski function} with parameter $\gamma \geq 0$ by:
\[
\cN(x) = \max_{i \in [N]} \big\{ v_i \cdot x - i \gamma \big\} \,,
\]
Note that 
\begin{equation} \label{eq:nemopt}
\cN^* \leq \cN(x^*) \leq - \frac{1}{\sqrt{N}} \,.
\end{equation}
In what follows, we restrict the query points to the unit ball for the function under consideration.

\subsection{The classical argument} \label{sec:NY}
We consider here the Nemirovski function with parameter $\gamma = 0$. Each gradient query reveals a single vector in the collection of the $v_i$, so after $N/2$ iterations one might know say $v_1,\ldots, v_{N/2}$, but the rest remain unknown (or in other words they remain random orthonormal vectors in $\mathrm{span}(v_1,\hdots, v_{N/2})^{\perp}$). 
Thus for any output $x$ that depends on only $N/2$ queries, one has $\mathbb{E} [ \cN(x) ] \geq 0$ (formally this inequality follows from Jensen's inequality and the tower rule). Thus, together with \eqref{eq:nemopt}, it follows that $\mathbb{E} [ \cN(x) - \cN^* ] \geq 1/\sqrt{N}$. In other words the best rate of convergence of sequential methods is $1/\sqrt{N}$, provided that $N \leq d$.

\subsection{The basic parallel argument} \label{sec:Nem94}
We consider here the Nemirovski function with parameter $\gamma = C \sqrt{\log(d)/d}$ for some large enough constant $C$ (more precisely that constant $C$ depends on the exponent in the $\mathrm{poly}(d)$ number of allowed queries per round). The key observation is as follows: Imagine that the algorithm has already discovered $v_1,\ldots,v_{i-1}$. Then for any set of $\mathrm{poly}(d)$ queries, with high probability with respect to the random draw of $v_i, \ldots, v_N$, one has that the inner product of any of those vectors with any of the queried points is in $[-\gamma/2, \gamma/2]$ (using both basic concentration of measure on the sphere, and a union bound). Thus the maximum in the definition of $\cN$ is attained at some index $\leq i$. This means that this set of $\mathrm{poly}(d)$ queries can only reveal $v_i$, and not any of the $v_j, j > i$. Thus after $N-1$ rounds we know that with high probability any output $x$ satisfies  
$
\cN(x) \geq v_N \cdot x - N \gamma \geq - (N+1) \gamma$ (since $v_N$ is a random direction orthogonal to $\mathrm{span}(v_1,\hdots,v_{N-1})$ and $x$ only depends on $v_1, \hdots, v_{N-1}$). 
Thus we obtain that the suboptimality gap is $\frac{1}{\sqrt{N}} - (N+1) \gamma$. Let us assume that
\begin{equation} \label{eq:basicond}
N^{3/2} \leq \frac{1}{2 \gamma} \,,
\end{equation} 
i.e., $N= \tilde{O}(d^{1/3})$ (since $\gamma = C \sqrt{\log(d)/d}$) so that $1/\sqrt{N} > (N+1) \gamma$. Then one has that the best rate of convergence with a highly parallel algorithm is $\Omega(1/\sqrt{N})$ (i.e., the same as with purely sequential methods).  

\subsection{The wall function} \label{sec:Wall}
Our new idea to improve upon Nemirovski's construction is to introduce a new random {\em wall function} $\cW$ (with parameter $\delta >0$), where the randomness come from $v_1, \hdots, v_N$. Our new random hard function, which we term {\em shielded-Nemirovski function}, is then defined by:
\[
f(x) = \max \left\{ \cN(x), \cW(x) \right\} \,.
\]
We construct the convex function $\cW$ so that one can essentially repeat the argument of Section \ref{sec:Nem94} with a smaller value of $\gamma$ (the parameter in the Nemirovski function), so that the condition \eqref{eq:basicond} becomes less restrictive and allows to take $N$ as large as $\tilde{O}(\sqrt{d})$.
\newline

Roughly speaking the wall function will satisfy the following properties:
\begin{enumerate}
\item The value of $\cW$ at $x^*$ is small, namely $\cW(x^*) \leq - \frac{1}{\sqrt{N}}$.
\item The value of $\cW$ at ``most" vectors $x$ with $\|x\| \geq \delta$ is large, namely $\cW(x) \geq \cN(x)$, and moreover it is does not depend on the collection $v_i$ (in fact at most points we will have the simple formula $\cW(x) = 2 \|x\|^{1+\alpha}$, for some small $\alpha$ that depends on $\delta$, to be defined later).
\end{enumerate}
The key argument is that, by property $2$, one can expect (roughly) that information on the random collection of $v_i's$ can only be obtained by querying points of norm smaller than $\delta$. This means that one can repeat the argument of Section \ref{sec:Nem94} with a smaller value of $\gamma$, namely $\gamma = \delta \cdot C \sqrt{\log(d)/d}$. In turn the condition \eqref{eq:basicond} now becomes $N = \tilde{O} \left( d^{1/3} / \delta^{2/3} \right)$. Due to convexity of $\cW$, there is a tension between property $1$ and $2$, so that one cannot take $\delta$ too small. We will show below that it is possible to take $\delta = \sqrt{N/d}$. In turn this means that the argument proves that $1/\sqrt{N}$ is the best possible rate, up to $N= \tilde{O}(\sqrt{d})$. 
\newline

The above argument is imprecise because the meaning of ``most" in property $2$ is unclear. A more precise formulation of the required property is as follows:
\begin{enumerate}
\item[$2'$] Let $x = w + z$ with $w\in V_i$ and $z \in V_i^{\perp}$ where $V_i = \mathrm{span}(v_1,\ldots, v_i)$. Assume that $\|z\| \geq \delta$, then the total variation distance between the conditional distribution of $v_{i+1}, \ldots, v_N$ given $\nabla \cW(x)$ (and $\cW(x)$) and the unconditional distribution is polynomially small in $d$ with high probability (here high probability is with respect to the realization of $\nabla \cW(x)$ and $\cW(x)$, see below for an additional comment about such conditional reasoning). Moreover if the argmax in the definition of $\cN(x)$ is attained at some index $>i$, then $\cW(x) \geq \cN(x)$.
\end{enumerate}
Given both property $1$ and $2'$ it is actually easy to formalize the whole argument. We do so by consdering a game between Alice, who is choosing the query points, and Bob who is choosing the random vectors $v_1, \hdots, v_N$. Moreover, to clarify the reasoning about conditional distributions, Bob will resample the vectors $v_i, \hdots, v_N$ at the beginning of the $i^{th}$ round of interaction, so that one explicitly does not have any information about those vectors given the first $i-1$ rounds of interaction. Then we argue that with high probability all the oracle answers remain consistent throughout this resampling process. See Appendix \ref{sec:LBapp} for the details. Next we explain how to build $\cW$ so as to satisfy property 1 and 2'.

\subsection{Building the wall} \label{sec:proofLB}
Let $h(x) = 2 \|x\|^{1+\alpha}$ be the basic building block of the wall. Consider the correlation cones:
\[
C_i = \left\{x \in \R^d : \left| v_i \cdot \frac{x}{\|x\|} \right| \geq C \sqrt{\frac{\log(d)}{d}} \right\} \,.
\]
Note that for any fixed query $x$, the probability (with respect to the random draw of $v_i$) that $x$ is in $C_i$ is polynomially small in $d$. We now define the wall $\cW$ as follows: it is equal to the function $h$ outside of the correlation cones and the ball of radius $\delta$, and it is extended by convexity to the rest of the unit ball. In other words, let $\Omega = \{x \in \R^d : \|x\| \in [\delta,1] \text{ and } x \not\in C_i \text{ for all } i \in [N]\}$, and
\[
\cW(x) = \max_{y \in \Omega} \left\{ h(y) + \nabla h(y) \cdot (x-y) \right\} \,.
\]

Let us first prove property 1:
\begin{lemma} \label{lem:prop1}
Let  $\alpha = \frac{1}{\log_2(1/\delta)} \leq 1$, and $\frac{\delta}{\log_2(1/\delta)} = 4 C  \sqrt{\frac{N \log(d)}{d}} + \frac{1}{\sqrt{N}}$. Then $\cW(x^*) \leq - \frac{1}{\sqrt{N}}$.
\end{lemma}

\begin{proof}
One has $\nabla h(y) = 2 (1+\alpha) \frac{y}{\|y\|^{1-\alpha}}$  and thus
\begin{equation} \label{eq:tangent}
 h(y) + \nabla h(y) \cdot (x-y) = - 2 \alpha \|y\|^{1+\alpha} + 2(1+\alpha) \frac{y \cdot x}{\|y\|^{1-\alpha}} \,.
\end{equation}
Moreover for any $y \in \Omega$ one has:
\[
|y \cdot x^*| \leq \frac{1}{\sqrt{N}} \sum_{i=1}^N |y \cdot v_i| \leq C \sqrt{\frac{N \log(d)}{d}}\cdot \|y\| \,.
\]
Thus for any $y \in \Omega$ we have:
\[
 h(y) + \nabla h(y) \cdot (x^*-y) \leq - 2 \alpha \delta^{1+\alpha} + 2 (1+\alpha) C \sqrt{\frac{N \log(d)}{d}} \,.
\]
The proof is straightforwardly concluded by using the values of $\alpha$ and $\delta$.
\end{proof}

Next we prove a simple formula for $\cW(x)$ in the context of property $2'$. More precisely we assume that $x = w + z$ with $w\in V_i$ and $z \in V_i^{\perp}$ with $z \not\in C_j$ for any $j > i$. Note that for any fixed $z$, the latter condition happens with high probability with respect to the random draw of $v_{i+1}, \ldots, v_N$.

\begin{lemma} \label{lem:2prime}
Let $x = w + z$ with $w\in V_i$ and $z \in V_i^{\perp}$ with $z \not\in C_j$ for any $j > i$. Then one has:
\[
\cW(x) = \max_{a, b \in \R_+ : a^2 + b^2 \in [\delta^2,1]} \left\{ - 2 \alpha (a^2 + b^2)^{\frac{1+\alpha}{2}} + 2 \frac{1+\alpha}{(a^2+b^2)^{\frac{1-\alpha}{2}}} \left( \max_{y \in\tilde{\Omega}_{a,b}, \|y\|=a} y \cdot w + b \|z\| \right) \right\} \,,
\]
where 
\[\tilde{\Omega}_{a,b} =  \left\{x \in V_i : \left|v_j\cdot \frac{x}{\|x\|}\right|\cdot \frac{a}{\sqrt{a^2+b^2}} < C\cdot \sqrt{\frac{\log(d)}{d}} \text{ for all } j \in [i]\right\}\]
and we use the convention that the maximum of an empty set is $-\infty$.
\end{lemma} 
\begin{proof}
Recall \eqref{eq:tangent}, and let us optimize over $y \in \Omega$ subject to $\|P_{V_i} y\| =a$ and $\|P_{V_i^{\perp}} y\| =b$ for some $a, b$ such that $a^2 + b^2 \in [\delta^2,1]$. Note that in fact there is an upper bound constraint on $a$ for such a $y$ to exists (for if the projection of $y$ onto $V_i$ is large, then necessarily $y$ must be in one of the correlation cones), which we can ignore thanks to the convention choice for the maximum of an empty set. Thus the only calculation we have to do is to verify that:
\[
\max_{y \in \Omega : \|P_{V_i} y\| =a \text{ and } \|P_{V_i^{\perp}} y\| =b} y \cdot x = \max_{y \in\tilde{\Omega}_{a,b}, \|y\|=a} y \cdot w + b \|z\| \,.
\]
Note that $y \cdot x = P_{V_i} y \cdot w + P_{V_i^{\perp}} y \cdot z$. Thus the right hand side is clearly an upper bound on the left hand side (note that $P_{V_i} y \in \tilde{\Omega}_{a,b}$). To see that it is also a lower bound take $y = y' + b \frac{z}{\|z\|}$ for some arbitrary $y' \in \tilde{\Omega}_{a,b}$ with $\|y'\| = a$, and note that $y \in \Omega$ (in particular using the assumption on $z$) with $\|P_{V_i} y\| =a$ and $\|P_{V_i^{\perp}} y\| =b$.
\end{proof}
The key point of the formula given in Lemma \ref{lem:2prime} is that it does not depend on $v_{i+1}, \hdots, v_N$. Thus when the algorithm queries the point $x$ and obtains the above value for $\cW(x)$ (and the corresponding gradient), the only information that it obtains is that $z \not\in C_j$ for any $j >i$. Since the latter condition holds with high probability, the algorithm essentially learns nothing (more precisely the conditional distribution of $v_{i+1}, \hdots, v_N$ only changes by $1/\mathrm{poly}(d)$ compared to the unconditional distribution).

Thus to complete the proof of property $2'$ it only remains to show that if $\|z\| \geq \delta$ and the argmax in $\cN(x)$ is attained at an index $>i$, then the formula in Lemma \ref{lem:2prime} is larger than $\cN(x)$. By taking $a=0$ and $b=\delta$ one obtains that this formula is equal to (using also the values assigned to $\alpha$ in Lemma \ref{lem:prop1}):
\[
- 2 \alpha \delta^{1+\alpha} + 2 \frac{1+\alpha}{\delta^{1-\alpha}} \delta \|z\| = - \alpha \delta + (1+\alpha) \|z\| \geq \|z\| \,.
\]
On the other hand one has (by assumption that the $\argmax$ index is $>i$)
\[
\cN(x) = \max_{j>i} \{v_j \cdot x - j \gamma\} \leq \|z\| \,.
\]
This concludes the proof of property $2'$, and in turn concludes the proof sketch of our lower bound.

\section{Upper bound} \label{sec:upper}

Here we present our highly parallel optimization procedure. Throughout this section we let $f : \R^d \rightarrow \R$ denote a differentiable $L$-Lipschitz function that obtains its minimum value at $x^* \in \R^d$ with $\|x^*\|_2 \leq R$. The main result of this section is the following theorem, which provides an $\tilde{O}(d^{1/3} / \epsilon^{2/3} )$-depth highly-parallel algorithm that computes an $\epsilon$-optimal point with high probability. 

\begin{theorem}[Highly Parallel Function Minimization]
\label{thm:parallel_min_main}
There is a randomized highly-parallel algorithm which given any differentiable L-Lipschitz $f : \R^d \rightarrow \R$ minimized at $x^*$ with $\|x^*\| \leq R$ computes with probability $1 - \nu$
a point $x \in \R^d$ with $f(x) - f(x^*) \leq \epsilon$ in depth $\tilde{O}(d^{1/3}(LR/\epsilon)^{2/3})$ and work $\tilde{O}(d^{4/3} (LR/\epsilon)^{8/3})$ where $\tilde{O}(\cdot)$ hides factors polylogarithmic in $d$,$\epsilon$,$L$,$R$, and $\nu^{-1}$.
\end{theorem}
\vspace{-2pt}
Our starting point for obtaining this result are the $O(d^{1/4} / \epsilon)$ depth highly parallel algorithms of \citep{DBW12}. This paper considers the convolution of $f$ with simple functions, e.g. Gaussians and uniform distributions, and shows this preserves the convexity and continuity of $f$ while improving the smoothness and thereby enables methods like accelerated gradient descent (AGD) to run efficiently. Since the convolved function can be accessed efficiently in parallel by random sampling, working with the convolved function is  comparable to working with the original function in terms of query depth (up to the sampling error). Consequently, the paper achieves its depth bound by trading off the error induced by convolution with the depth improvements gained from stochastic variants of AGD. 

To improve upon this bound, we apply a similar approach of working with the convolution of $f$ with a Gaussian. However, instead of applying standard stochastic AGD we consider accelerated methods which build a more sophisticated model of the convolved function in parallel. Instead of using random sampling to approximate only the gradient of the convolved function, we obtain our improvements by using random sampling to glean more local information with each highly-parallel query and then use this to minimize the convolved function at an accelerated rate.

To enable the use of these more sophisticated models we develop a general acceleration framework that allows us to leverage any subroutine for approximate minimization a local model/approximate gradient computations into an accelerated minimization scheme. We believe this framework is of independent interest, as we show that we can analyze the performance of this method just in terms of simple quantities regarding the local model. This framework is discussed in Section~\ref{sec:acceleration_framework} and in Appendix~\ref{sec:applications} where we show how it generalizes multiple previous results on near-optimal acceleration.

Using this framework, proving Theorem~\ref{thm:parallel_min_main} reduces to showing that we can minimize high quality local models of the convolved function. Interestingly, it is possible to nearly obtain this result by simply random sampling to estimate all derivatives up to some order $k$ and then use this to minimize a regularized $k$-th order Taylor approximation to the function. Near-optimal convergence for such methods under Lipschitz bounds on the $k$-th derivatives were recently given by \citep{Gasnikov18, JWZ18, BJLLS8} (and follow from our framework). This approach can be shown to give a highly-parallel algorithm of depth $\tilde{O}(d^{1/3 + c} / \epsilon^{2/3} )$ for any $c > 0$ (with an appropriately large $k$).  Unfortunately, the work of these methods is $O(d^{\mathrm{poly}(1/c)})$ and expensive for small $c$.

To overcome this limitation, we leverage the full power of our acceleration framework and instead show that we can randomly sample to build a model of the convolved function accurate within a ball of sufficiently large radius. In Section~\ref{sec:highly_parallel_optimization} we bound this quality of approximation and show that this local model can be be optimized to sufficient accuracy efficiently. By combining this result with our framework we prove Theorem~\ref{thm:parallel_min_main}. We believe this demonstrates the utility of our general acceleration scheme and we plan to further explore its implications in future work.

\subsection{Acceleration framework}
\label{sec:acceleration_framework}

Here we provide a general framework for accelerated convex function minimization.
Throughout this section we assume that there is a twice-differentiable convex function $g : \R^d \rightarrow \R$ given by an \emph{approximate proximal step oracle} and an \emph{approximate gradient oracle} defined as follows.

\begin{definition}[Approximate Proximal Step Oracle] 
	\label{def:prox_oracle}
	Let $\omega : \R_+ \rightarrow \R_+$ be a non-decreasing function, $\delta \geq 0$, and $\alpha \in [0,1)$. We call $\oracleprox$ an \emph{$(\alpha, \delta)$-approximate $\omega$-proximal step oracle} for $g : \R^d \rightarrow \R$ if, for all $x \in \R^d$, when queried at $x \in \R^d$ the oracle returns $y = \oracleprox(x)$ such that
	\begin{equation}
	\label{eq:oracle_prox_condition}
	\|\nabla g(y)+ \omega(\|y - x\|) (y - x) \| \leq  \alpha \cdot \omega(\|y - x\|) \|y - x\| + \delta
	~.
	\end{equation}
\end{definition}

\begin{definition}[Approximate Gradient Oracle] 
	\label{def:noisy_gradient}
	We call $\oraclenoiseygrad$ an \emph{$\delta$-approximate gradient oracle} for $g : \R^d \rightarrow \R$ if when queried at $x \in \R^d$ the oracle returns  $v = \oraclenoiseygrad(x)$ such that $\|v - \grad g(x)\| \leq \delta$. 
\end{definition}

We show that there is an efficient accelerated optimization algorithm for minimizing $g$ using only these oracles. Its performance is encapsulated by the following theorem.

\begin{theorem}[Acceleration Framework]
\label{thm:main_acceleration}
	Let $g : \R^d \rightarrow \R$ be a convex twice-differentiable function minimized at $x^*$ with $\|x^*\| \leq R$, $\epsilon > 0$, $\alpha \in [0, 1)$, and $\gamma \geq 1$ such that $128 \alpha \gamma^2 \leq 1$. Further, let $\omega : \R_+ \rightarrow \R_+$ be a monotonically increasing continuously differentiable function with $0 < \omega'(s) \leq \gamma\cdot \omega(s) / s$ for all $s > 0$.  There is an algorithm which for all $k$ computes a point $y_k$ with
\[
g(y_{k})-g^* \leq 
\max \Bigg\{\epsilon ~ , ~
\frac{32\cdot\omega\left(\frac{40\|x^{*}\|}{k^{3/2}}\right) \|x^{*}\|^2}{k^{2}}\Bigg\}
\]
using  $k (6+\log_2[10^{20} \gamma^6 R^2 \cdot \omega(10^5 \gamma^2 R) \cdot \epsilon^{-1}])^2$ queries to a $(\alpha, \delta)$-approximate $\omega$-proximal step oracle for $g$ and a $\delta$-approximate gradient oracle for $g$ provided that both
$\delta \leq \epsilon/[10^{20} \gamma^2 R]$ and $\epsilon \leq 10^{20} \gamma^4 R^3  \omega(80 R)$.
\end{theorem}

This theorem generalizes multiple accelerated methods (up to polylogarithmic factors) and sheds light on the rates of these methods (See Appendix~\ref{sec:applications} for applications). For example, choosing $\omega(x) \defeq \frac{L}{2}$ and $\oracleprox(x) = x - \frac{1}{L} \grad f(x)$ recovers standard accelerated minimization of $L$-smooth functions, choosing $\omega(x) \defeq \frac{L}{2}$ and $\oracleprox(x) \approx \argmin_{y} g(y) + \frac{L}{2} \|y - x\|^2$ recovers a variant of approximate proximal point \citep{FrostigGKS15} and Catalyst \citep{LinMH15}, and choosing $\omega(x) \defeq \frac{L_p \cdot (p + 1)}{p!} x^{p -1}$ and $\oracleprox(x) = \argmin_{y}\; g_p(y;x) + \frac{L_p}{p!} \|y - x\|^{p + 1}$ where 
$g_p(y;x)$ is the value of the $p$'th order Taylor approximation of $g$ about $x$ evaluated at $y$ recovers highly smooth function minimization \citep{MS13, Gasnikov18, JWZ18, BJLLS8}.

We prove Theorem~\ref{thm:main_acceleration} by generalizing an acceleration framework due to \citep{MS13}. This framework was recently used by several results to obtain near-optimal query complexities for minimizing highly smooth convex functions \citep{Gasnikov18, JWZ18, BJLLS8}. In Section~\ref{sec:algorithm} we provide a variant of this general framework that is amenable to the noise induced by our oracles. In Section~\ref{sec:convergence_with_oracle} we show how to instantiate our framework using the oracles assuming a particular type of line search can be performed. In Section~\ref{sec:put-together} we then prove the Theorem~\ref{thm:main_acceleration}. The algorithm for and analysis of line search is deferred to Appendix~\ref{sec:implementation}.

\subsection{Highly parallel optimization}
\label{sec:highly_parallel_optimization}

With Theorem~\ref{thm:main_acceleration} in hand, to obtain our result we need to provide, for an appropriate function $\omega$, a highly parallel implementation of an approximate proximal step oracle and an approximate gradient oracle for a function that is an $O(\epsilon)$ additive approximation $f$. As with previous work \citep{DBW12,SBBML18} we consider the convolution of $f$ with a Gaussian of covariance $r^{2}\cdot I_d$ for $r > 0$ we will tune later. Formally, we define $g : \R^d \rightarrow \R$ for all $x \in \R^d$ as 
\[
g(x) \defeq \int_{\R^{d}} \gamma_r(y) f(x-y) dy 
\quad \text{where} \quad
\gamma_r (x) \defeq \frac{1}{(\sqrt{2\pi}r)^d}\exp\left(-\frac{\|x\|^2}{2r^2} \right)
\]
It is straightforward to prove (See Section~\ref{sec:convolution}) the following standard facts regarding $g$. 

\begin{lem}
\label{lem:convolution_err} 
The function $g$ is convex, $L$-Lipschitz, and satisfies both $|g(y)-f(y)|\leq\sqrt{d}\cdot Lr$ and $\nabla^{2}g(y)\preceq (L/r) \cdot I_d$ for all $y \in \R^d$. 
\end{lem}

Consequently, to minimize $f$ up to $\epsilon$ error, it suffices to minimize $g$ to $O(\epsilon)$ error with $r=O(\frac{\epsilon}{\sqrt{d}L})$. In the remainder of this section we simply show how to provide highly parallel implementations of the requisite oracles to achieve this by Theorem~\ref{thm:main_acceleration}.

Now, as we have discussed, one way we could achieve this goal would be to use random sampling to approximate (in parallel) the $k$-th order Taylor approximation to $g$ and minimize a regularization of this function to implement the approximate proximal step oracle. While this procedure is depth-efficient, its work is quite large. Instead, we provide a more work efficient application of our acceleration framework. To implement a query to the oracles at some point $c \in \R^d$ we instead simply take multiple samples from $\gamma_r(x -c)$, i.e. the normal distribution with covariance $r^{2}I_{d}$ and mean $c$, and use these samples to build an approximation to the gradient field of $g$. The algorithm for this procedure is given by Algorithm~\ref{alg:apx} and carefully combines the gradients of the sampled points to build a model with small bias and variance. By concentration bound and $\epsilon$-net argument, we can show that
of Algorithm~\ref{alg:apx} outputs a vector field $v : \R^d \rightarrow \R^d$ that is an uniform approximation of $\nabla g$ within a small ball (See Section~\ref{sec:noisy_gradient} for the proof.)

\begin{algorithm2e}[H]
	
\caption{Compute vector field approximating $\nabla g$\label{alg:apx}}
\SetAlgoLined

\textbf{Input}: Number of samples $N$, radius $r > 0$, error parameter $\eta \in (0, 1)$, center $c \in \R^d$. 

Sample $x_{1},x_{2},\cdots$, $x_{N}$ independently from $\gamma{}_{r}(x-c)$.

\Return $v : \R^d \rightarrow \R^d$ defined for all $y \in \R^d$ by 
\[
v(y)=\frac{1}{N}\sum_{i=1}^{N}\frac{\gamma_{r}(y-x_{i})}{\gamma_{r}(c-x_{i})}\cdot\nabla f(x_{i})\cdot\chi((x_{i}-c)^{\top}(y-c))\cdot1_{\|x_{i}-c\|\leq(\sqrt{d}+\frac{1}{\eta})r}
\]
where $\chi(t) \defeq 0$, if $|t|\geq r^{2}$, $\chi(t) \defeq 1$ if $|t|\leq\frac{r^{2}}{2}$ and $\chi(t) \defeq 
2-\frac{2|t|}{r^{2}}$ otherwise.
\end{algorithm2e}
\begin{lem}[Uniform Approximation]
	\label{lem:uniform-approximation}
	Algorithm~\ref{alg:apx} outputs vector field $v : \R^d \rightarrow \R$ such that for any $\delta \in (0, \frac{1}{2})$ with probability at least $1 - \delta$ the following holds
	\[
	\max_{y:\|y-c\|\leq\frac{\eta}{4}r}\|v(y)-\nabla g(y)\|\leq5L\cdot\exp\left(-\frac{1}{2\eta^{2}}\right)+\frac{8L}{\sqrt{N}}\sqrt{\frac{d}{\eta^{2}}\log(9d)+\log\frac{1}{\delta}}
	\]
	Consequently, for any $\epsilon \in [0,1]$, $N=\mathcal{O}([d\log d\log(\frac{1}{\epsilon})+\log(\frac{1}{\delta})]\epsilon^{-2})$, and 
	$\eta=\frac{1}{2\sqrt{\log\left(\frac{10}{\epsilon}\right)}}$ this yields that $\max_{y:\|y-c\|\leq\tilde{r}}\|v(y)-\nabla g(y)\|\leq L\cdot\epsilon$ where $\tilde{r}=\frac{r}{8\sqrt{\log\left(\frac{10}{\epsilon}\right)}} $. 
\end{lem}

This lemma immediately yields that we can use Algorithm~\ref{alg:apx} to implement a highly-parallel approximate gradient oracle for $g$. Interestingly, it can also be leveraged to implement a highly-parallel approximate proximal step oracle. Formally, we show how to use it to find $y$ such that
\begin{equation}
\nabla g(y)+\omega(\|y-x\|)\cdot(y-x) \approx 0
\quad \text{where} \quad 
\omega(s) \defeq \frac{4Ls^{p}}{\tilde{r}^{p+1}}\label{eq:Lip_omega}
\end{equation}
for some $p$ to be determined later. Ignoring logarithmic factors and supposing for simplicity that $L, R \leq 1$, Theorem~\ref{thm:main_acceleration} shows that by invoking this procedure $\tilde{O}(k) \approx d^{\frac{p + 1}{3p + 4}} \epsilon^{\frac{-2p-2}{3p + 4}}$ times we could achieve function error on the order 
\[
\omega(1/k^{3/2}) / k^2 \approx \tilde{r}^{- (p + 1)} k^{- \frac{3p + 4}{2}}
\approx d^{\frac{p + 1}{2}} \epsilon^{- (p + 1)} k^{- \frac{3p + 4}{2}} \approx \epsilon
\]
and therefore achieve the desired result by setting $p$ to be polylogarithmic in the problem parameters. 

Consequently, we simply need to find $y$ satisfying \eqref{eq:Lip_omega}. The algorithm that achieves this is Algorithm~\ref{alg:apx-2}  which essentially performs gradient descent on
\begin{equation}
g(y)+\Phi(\|y-c\|)
\quad \text{where} \quad 
\Phi(s)=\int_{0}^{s}\omega(t)\cdot t\, dt \label{eq:Lip_Phi} ~.
\end{equation}

The performance of this algorithm is given by Theorem~\ref{thm:optimization_oracle} in Section~\ref{sec:approx_step_implement}.  Combined with all of the above it proves Theorem~\ref{thm:parallel_min_main}, see Section~\ref{sec:parallel_complexity} for the details.

\begin{algorithm2e}[H]
		
		\caption{Approximate minimization of  $g(y)+\Phi(\|y-c\|)$\label{alg:apx-2}}
		
		\SetAlgoLined
		
		\textbf{Input}: center $c \in \R^d$, accuracy $\epsilon$, inner radius $\tilde{r}=\frac{r}{8\sqrt{\log\left(\frac{60}{\epsilon}\right)}}$, 
		and step size $h=\frac{\tilde{r}}{48p\sqrt{d}L}$.
		
		Use Algorithm \ref{alg:apx} to find a vector field $v$ such that
		$
		\max_{y:\|y-c\|\leq\tilde{r}}\|v(y)-\nabla g(y)\|\leq L\cdot\frac{\epsilon}{6}
		 $.		
		 
		$y\leftarrow c$.
		
		\For{$i=1,2,\cdots\infty$}{
			
			$\delta_{y}=v(y)+\omega(\|y-c\|)\cdot(y-c)$ where $\omega$ is
			defined by \eqref{eq:Lip_omega} with $p \geq 1$.
			
			\lIf*{$\|\delta_{y}\|\leq L\cdot\frac{5\epsilon}{6}$}{\Return $y$} \lElse{$y=y-h\cdot\delta_{y}$}

		}

	\end{algorithm2e}

\section*{Acknowledgements}

We thank Ankit Garg, Robin Kothari, Praneeth Netrapalli, and Suhail Sherif for helpful feedback and pointing out an earlier issue in the formula for $\tilde{\Omega}_{a,b}$.

\bibliographystyle{plainnat}
\bibliography{bib}

\appendix

\newpage
\section{Further details on the lower bound} \label{sec:LBapp}
\global\long\def\set#1{\mathcal{#1}}

In this section we prove our main lower bound theorem from Section~\ref{sec:lower} repeated below:

\lowerbound*

To prove Theorem \ref{thm:lower_bound}, we consider the following
game between the algorithm (player A) issuing the queries, and the adversary (player B) building the hard {\em shielded Nemirovski} function $f$ (as defined in Section \ref{sec:Wall} and Section \ref{sec:proofLB}), i.e., player B chooses the orthonormal vectors in the definition of $f$. To make explicit the dependency of the shielded Nemirovski function on the choice of the orthonormal vectors $v_1, \cdots, v_N$, we denote it by $f^{v_1, \cdots, v_N}$ (with similar notation for the Nemirovski function $\cN$ and the wall function $\cW$). We restrict our attention to a deterministic player A and randomized player B, which is without loss of generality thanks to the minimax theorem. The game has $N$ iterations, and at each iteration $t$, players A and B
maintain a common set of orthonormal vectors $\set V_{t}=\{v_{1},v_{2},\cdots,v_{t}\}$,
and common sets of vectors $\set Q_{1},\set Q_{2},\cdots,\set Q_{t}$
where initially $\set Q_{0}=\emptyset$. At each iteration,
\begin{enumerate}
\item Simultaneously:
\begin{enumerate}
\item Player A queries a set of $Q$ points $\set Q_{t}=\{z_{t}^{(1)},\cdots,z_{t}^{(Q)}\}$
inside the unit ball. 
\item Player B randomly sample $N-t+1$ orthonormal vectors $v_{t}^{(t)},v_{t}^{(t+1)},\cdots,v_{t}^{(N)}$
from $\text{span}(\set V_{t-1})^{\bot}$.
\end{enumerate}
\item Player B returns $f^{v_{1},v_{2},\cdots,v_{t},v_{t}^{(t+1)},\cdots,v_{t}^{(N)}}(x)$
and $\nabla f^{v_{1},v_{2},\cdots,v_{t},v_{t}^{(t+1)},\cdots,v_{t}^{(N)}}(x)$
to player $A$ for every $x\in\set Q_{t}$ where $v_{t} :=v_{t}^{(t)}$.
\end{enumerate}
Note that at each iteration Player B answers the query with a different function, however we will show that in fact with high probability all the given answers are consistent with the final function. More precisely let us introduce the high probability event under which we will carry out the proof. We say that Player B wins the game if the following holds:
\[
\forall t\in[N],z\in\set Q_{t},s_{1},s_{2}\geq t,\left|\langle x,v_{s_{1}}^{(s_{2})}\rangle\right|< \sqrt{\frac{C \log d}{d}}\cdot\left\Vert P_{V_{t-1}^{\perp}}x\right\Vert \,.
\]
\begin{lemma}
\label{lem:win_prob}
Let $\rho \in (0,1)$. Assume $N\leq\frac{d}{2}$ and let $C = 12 + 4 \log_d(Q/\rho)$. Then player B wins with
probability at least $1- \rho$.
\end{lemma}

\begin{proof}
For any $s_{1}$ and $s_{2}$, we note that $v_{s_{1}}^{(s_{2})}$
follows the uniform distribution on the unit sphere restricted on
the subspace $V_{s_{1}-1}^{\perp}$. For any $x\in\set Q_{t}$, we
have that
\[
\langle x,v_{s_{1}}^{(s_{2})}\rangle=\langle P_{V_{s_{1}-1}^{\perp}}x,P_{V_{s_{1}-1}^{\perp}}v_{s_{1}}^{(s_{2})}\rangle.
\]
By \cite[Lemma 2.2]{Bal97},
we have that
\[
\mathbb{P}_{v_{s_{1}}^{(s_{2})}}\left(\left|\langle P_{V_{s_{1}-1}^{\perp}}x,P_{V_{s_{1}-1}^{\perp}}v_{s_{1}}^{(s_{2})}\rangle\right|\geq t\cdot\|P_{V_{s_{1}-1}^{\perp}}x\|_{2}\right)\leq2\exp\left(-\dim V_{s_{1}-1}^{\perp}\cdot\frac{t^{2}}{2}\right).
\]
Since $\dim V_{s_{1}-1}^{\perp}=d-s_{1}+1\geq d-N\geq\frac{d}{2}$
and $\|P_{V_{s_{1}-1}^{\perp}}x\|_{2}\leq\|P_{V_{t-1}^{\perp}}x\|_{2}$
(using $t\leq s_{1}$), we have that
\[
\mathbb{P}\left(\left|\langle x,v_{s_{1}}^{(s_{2})}\rangle\right|>\sqrt{\frac{C\log d}{d}}\cdot\left\Vert P_{V_{t-1}^{\perp}}x\right\Vert _{2}\right)\leq2\exp\left(-\frac{d}{2}\cdot\frac{1}{2}\cdot\frac{C\log d}{d}\right)=2d^{-\frac{C}{4}}.
\]
Taking union bound over at most $N^{2}$ pairs of $v_{i}^{(j)}$ and
$NQ$ many $x$, we have that player B wins with
probability at least $1-Q\cdot d^{3-\frac{C}{4}}$, which concludes the proof.
\end{proof}

Next we show that if Player B wins the game, then indeed all answers are consistent with the final function.

\begin{lemma}
\label{lem:game_consistence}Assume player B wins the game and that $\gamma = 2 \delta \sqrt{\frac{C \log(d)}{d}}$.
Then, for all $t\in[N]$ and all $x\in\set Q_{t}$, we have that
\begin{equation}
f^{v_{1},v_{2},\cdots,v_{t},v_{t}^{(t+1)},\cdots,v_{t}^{(N)}}(x)=f^{v_{1},v_{2},\cdots,v_{t},v_{t+1},\cdots,v_{N}}(x)\label{eq:F_consistence}
\end{equation}
and that
\begin{equation}
\nabla f^{v_{1},v_{2},\cdots,v_{t},v_{t}^{(t+1)},\cdots,v_{t}^{(N)}}(x)=\nabla f^{v_{1},v_{2},\cdots,v_{t},v_{t+1},\cdots,v_{N}}(x)\label{eq:F_consistence2}
\end{equation}
\end{lemma}

\begin{proof}
Fix any $t\in[N]$ and any $x \in\set Q_{t}$. Write $x=w+z$ with $w \in V_{t-1}$ and $z \in V_{t-1}^{\perp}$. Since player B wins,
we have that
\[
\left|\langle z,v_{t}^{(s)}\rangle\right|\leq\sqrt{\frac{C\log d}{d}}\cdot\left\Vert z\right\Vert \,,
\]
for all $s\geq t$. Lemma \ref{lem:2prime} shows that 
\begin{equation} \label{eq:wwwapp1}
\cW^{v_{1},v_{2},\cdots,v_{t},v_{t}^{(t+1)},\cdots,v_{t}^{(N)}}(x) = \cW^{v_{1},v_{2},\cdots,v_{t},v_{t+1},\cdots,v_{N}}(x) \,.
\end{equation}

Moreover the equations following Lemma \ref{lem:2prime} show that \eqref{eq:wwwapp1} also holds for the function $f$ itself provided that $\|z\| \geq \delta$ (indeed, as discussed there if the argmax index in the definition of the Nemirovski function is attained at an index $\geq t$ then in fact $f(x) = \cW(x)$, and otherwise the Nemirovski function value itself does not depend on $v_{t}^{(t+1)},\cdots,v_{t}^{(N)}$).

Thus we only need to consider the case where $\|z\| \leq \delta$. In this case we prove that \eqref{eq:wwwapp1} also holds for the Nemirovski function (and thus it also holds for $f$). Indeed for any $s > t$
\begin{align*}
\langle v_{s},x\rangle-\gamma\cdot s & =\langle v_{s},z\rangle -\langle v_{t}, z\rangle+\langle v_{t},x\rangle-\gamma\cdot s\\
 & \leq 2 \sqrt{\frac{C \log d}{d}}\cdot \delta +\langle v_{t},x\rangle-\gamma\cdot s\\
 & \leq\langle v_{t},x\rangle-\gamma\cdot t \,,
\end{align*}
where the last inequality uses that $\sqrt{\frac{C \log d}{d}} = \frac{\gamma}{2 \delta}$. This concludes the proof of  \eqref{eq:F_consistence}. For \eqref{eq:F_consistence2} we simply note that \eqref{eq:F_consistence} remains true for infinitesimal perturbations of $x$.
\end{proof}

Finally we show that no queried point could have a suboptimal gap smaller than $o(1/\sqrt{N})$.

\begin{lemma}
Assume player $B$ wins and that $\log(N) N \sqrt{\frac{C \log(d)}{d}} \leq \frac14$. Then, for all $t \in [N]$ and all $x \in \set Q_t$, we have that
\[
f(x) - f^* \geq \frac{1}{4 \sqrt{N}} \,.
\]
\end{lemma}

\begin{proof}
First we claim that
\[
f(x) - f^* \geq \frac{1}{\sqrt{N}} - \sqrt{\frac{C \log(d)}{d}} - \gamma N \,.
\]
This follows from \eqref{eq:nemopt}, Lemma \ref{lem:prop1}, and the fact that:
\[
f(x) \geq \cN(x) \geq \langle v_N, x \rangle - \gamma N \,.
\]
Next recall from Lemma \ref{lem:game_consistence} that we take $\gamma = 2 \delta \sqrt{\frac{C \log(d)}{d}}$, and from Lemma \ref{lem:prop1} that $\frac{\delta}{\log_2(1/\delta)} = 4 \sqrt{\frac{C N \log(d)}{d}} + \frac{1}{\sqrt{N}} \leq \frac{2}{\sqrt{N}}$ where the inequality follows from the assumption on $N$. In particular we have $\delta \leq \frac{\log(N/2)}{\sqrt{N}}$. Thus:
\[
f(x) - f^* \geq \frac{1}{\sqrt{N}} - \sqrt{\frac{C \log(d)}{d}} \left(1 + 2 \log(N/2) \sqrt{N} \right) \geq \frac{1}{4 \sqrt{N}} \,,
\]
where the second inequality follows from the assumption on $N$.
\end{proof}
\section{Acceleration with approximate proximal step oracles}
\label{sec:acceleration}

Here we provide the proofs associated with Section~\ref{sec:acceleration_framework} and prove Theorem~\ref{thm:main_acceleration}. Our proof is split into several parts. In Section~\ref{sec:algorithm} we provide the acceleration framework we leverage, in Section~\ref{sec:convergence_with_oracle} we show how to instantiate the framework using our oracles, and in Section~\ref{sec:put-together} we then prove Theorem~\ref{thm:main_acceleration}. This analysis relies on a line search result deferred to Appendix~\ref{sec:implementation}.

\subsection{Framework}
\label{sec:algorithm}

In this section we present the general acceleration framework based on \cite{MS13} which we leverage to achieve our result. This acceleration framework is given by Algorithm~\ref{alg:framework} and is a noise-tolerant analog of the one present in \cite{BJLLS8}. The framework maintains points  $x_k$ and $y_k$ in each iteration $k$. To compute the next point, a careful convex combination of them is chosen, denote $\tilde{x}_k$, and the next $y_{k  + 1}$ is chosen a point that has similar properties to the result of an approximate proximal step oracle and the next $x_{k + 1}$ is then the result of moving from $x_{k}$ in the direction of $\grad g(y_{k + 1})$. Here we provide general results regarding the iterates in the general setting of Algorithm~\ref{alg:framework}. In the next section we show how to implement the framework and ultimately bound the error.

\begin{algorithm2e} [ht!]
	\caption{Acceleration Framework}\label{alg:framework}
	\label{alg:accel_framework}
	\SetAlgoLined
	\textbf{Input:} $x_0=y_0 = 0_d$, $\sigma \in (0, 1)$, $A_0 = 0$, $K > 0$
	
	\For{ $k = 0, \ldots , K - 1$ }{
		Compute $\lambda_{k+1} > 0$ and $y_{k+1} \in \R^d$ such that for 
		\[
		a_{k+1} \defeq \frac{1}{2}\left[\lambda_{k+1}+\sqrt{\lambda_{k+1}^2+4\lambda_{k+1}A_k}\right] 
		\text{ , } 
		A_{k+1} \defeq A_k+a_{k+1}
		\text{ , } 
		\tilde{x}_k \defeq \frac{A_k}{A_{k + 1}}y_k + \frac{a_{k+1}}{A_{k+1}} x_k,
		\]
		the following condition holds
		\begin{equation} \label{eq:framework}
		\|\lambda_{k+1}\nabla g(y_{k+1}) + y_{k+1}-\tilde{x}_k\|  \leq \sigma \|y_{k+1}-\tilde{x}_k\|+\lambda_{k+1}\delta \, .
		\end{equation} 
		
		Compute $x_{k+1}$ such that the following holds
		\begin{equation}\label{eqn: update_x}
		\|x_{k+1}-(x_k-a_{k+1}\nabla g(y_{k+1}))\| \leq a_{k+1} \delta 
		\end{equation}
	}
	\textbf{return} $y_{K}$ 
\end{algorithm2e}

\begin{rem}
\label{rem:a_k}
The definition of $a_{k+1}$ was chosen such that $\lambda_{k + 1} A_{k + 1} = a_{k+1}^2$. To see this, note that $a_{k + 1}$ is a solution to $a_{k+1}^2 - \lambda_{k + 1} a_{k+1} - \lambda_{k + 1} A_k = 0$, which is equivalent as $A_{k + 1} = A_k + a_k$.
\end{rem}

In the following theorem we give a general bound for the quality of the iterates in Algorithm~\ref{alg:framework}.

\begin{theorem}[Framework Convergence]
	\label{thm:convergence}
	Algorithm~\ref{alg:framework} above gives for all $k \geq 1$ that
	\[
	A_{k} \left[ g(y_{k}) - g^* \right] + \frac{1}{2} \|x_{k}-x^*\|^2 
	+ \sum_{i \in [k]} \frac{(1 - \sigma)A_{i}}{2\lambda_{i}}\|y_{i} -\tilde{x}_{i - 1}\|^2
	\leq \frac{1}{2}\|x^*\|^2 
	+ \delta_{k}
	\]
	where
	\[
	\delta_{k} = \delta \sum_{i \in [k]} a_{i}\|x_{i}-x^*\|+\frac{\delta^2}{2(1-\sigma)}\sum_{i \in [k]}a_i^2 ~.
	\]
\end{theorem}
\begin{proof}
	Let $\Delta_{k + 1} \defeq x_{k + 1} - (x_k - a_{k + 1} \grad g(y_{k + 1}))$, $r_k  \defeq \frac{1}{2} \norm{x_k - x^*}^2$, and $\epsilon_k  \defeq g(y_k) - g^*$ so 
	\begin{align*}
	\frac{1}{2}  \norm{x_{k + 1} - x^* - \Delta_{k + 1}}^2
	&= r_k +  a_{k + 1} \grad g(y_{k + 1})^\top (x^* - x_k) + \frac{a_{k + 1}^2}{2} \norm{ \grad g(y_{k + 1}) }^2 ~.
	\end{align*}
	Now, since
	\[
	x_k 
	= y_k  + \frac{A_{k + 1}}{a_{k + 1}} (\tilde{x}_k - y_k) 
	=  y_{k + 1} + \frac{A_{k + 1}}{a_{k + 1}} (\tilde{x}_k - y_{k + 1}) 
	+ \frac{A_{k}}{a_{k + 1}} (y_{k + 1} - y_k) 
	\]
	and by convexity $g(z)  \geq g(y_{k + 1}) + \grad g(y_{k + 1})^\top (z - y_{k + 1})$ for all $z$ we have
	\begin{align*}
	a_{k + 1} \grad g(y_{k + 1})^\top (x^* - x_k) 
	&\leq 
	A_{k + 1} \grad g(y_{k + 1})^\top (y_{k +  1} - \tilde{x}_k)
	+ A_{k} \epsilon_{k} - A_{k + 1}  \epsilon_{k + 1}  ~.
	\end{align*}
	Combining these inequalities and applying Cauchy Schwarz yields
	\begin{align*}
	r_{k + 1}
	&= \frac{1}{2} \norm{x_{k + 1} -x^* - \Delta_{k + 1}}^2 
	+ \Delta^\top_{k + 1} (x_{k + 1} -x^* - \Delta_{k + 1}) + \frac{1}{2} \norm{\Delta_{k + 1}}^2
	\\
	&\leq 
	r_{k} + 
	A_{k + 1} \grad g(y_{k + 1})^\top (y_{k +  1} - \tilde{x}_k)
	+ A_{k} \epsilon_{k} - A_{k + 1}  \epsilon_{k + 1}
	+ \frac{a_{k + 1}^2}{2} \norm{ \grad g(y_{k + 1}) }^2\\
	& \quad
	+ \norm{\Delta_{k + 1}} \norm{x_{k + 1} - x^*}
	\end{align*}
	Now rearranging \eqref{eq:framework} and applying $(a+b)^2 \leq (1 + t) a^2+ (1 + t^{-1})b^2$ for $t = \frac{1 - \sigma}{\sigma}$
	yields  
	\[2\lambda_{k+1}\nabla g(y_{k+1})^\top(y_{k+1}-\tilde{x}_k)+\lambda_{k+1}^2\|\nabla g(y_{k+1})\|^2\leq -(1- \sigma)\|y_{k+1}-\tilde{x}_k\|^2+(1-\sigma)^{-1}\lambda_{k+1}^2\delta^2\]
	Combining with the facts that $\lambda_kA_k = a_k^2$ and $\norm{\Delta_{k + 1}} \leq  a_{k+1} \delta $ yields
	\begin{align*}
	r_{k+1}+A_{k+1}\epsilon_{k+1}+\frac{(1 - \sigma)A_{k+1}}{2\lambda_{k+1}}\|y_{k+1} -\tilde{x}_k\|^2&\leq r_k+A_k\epsilon_k +a_{k+1} \delta \|x_{k+1}-x^*\|+\frac{\delta^2}{2(1-\sigma)}a_{k+1}^2
	\end{align*}
	Summing over $k$ and using that $A_0 = 0$ and $x_0 = 0$ yields the result.
\end{proof}

Next we show that for sufficiently small $\delta$, the error in Theorem~\ref{thm:convergence} is increased by only a constant factor. This will allow us to apply Theorem~\ref{thm:convergence} when $\delta \neq 0 $.

\begin{lem}[Error Tolerance]
	\label{err:bound}	
	Algorithm~\ref{alg:framework} with $\delta \leq c\sqrt{1-\sigma} \|x^*\| / A_{K}$ for some $c, K \geq 0$ gives that $\delta_k \leq c(1+3c)\|x^*\|^2$. Consequently, if $c \leq \frac{1}{4}$ then for all $k \in [K]$
	\begin{equation}
	A_{k} \left[ g(y_{k}) - g^* \right] + \frac{1}{2} \|x_{k}-x^*\|^2 
	+ \sum_{i \in [k]} \frac{(1 - \sigma)A_{i }}{2\lambda_{i}}\|y_{i} -\tilde{x}_{i-1}\|^2
	\leq \|x^*\|^2 
	\label{eq:converge_rate}
	\end{equation}
	In particular, this implies that taking $\delta \leq \frac{\|x^*\|}{\mu\cdot A_{K}}$ for $\mu \defeq \frac{4\sqrt{2}}{\sqrt{1-\sigma}}$ then $\|x_{k}-x^{*}\|\leq2\|x^{*}\|$. Furthermore, we have that either $g(y_{k})\leq g^*+\epsilon$ or $A_{k}\leq\frac{\|x^*\|^2}{\epsilon}$. 
\end{lem}

\begin{proof}
	Theorem~\ref{thm:convergence}, the assumption on $\delta$, $\sigma\in[0,1)$ and $A_{K} = \sum_{i \in [K]} a_i$ yield that for all $k \in [K]$
	\[
	\frac{1}{2}\|x_{k}-x^*\|^2 
	\leq  
	\frac{1}{2} \|x^*\|^2+ c \|x^*\| \max_{i \in [K]} \|x_{i}-x^*\|+\frac{c^2}{2} \|x^*\|^2
	\]
	Since this holds for all $k \in [K]$ it clearly holds for $k \in \argmax_{i \in [K]} \|x_i - x^*\|$ and therefore 
	\[
	 \max_{i \in [K]} \|x_{i}-x^*\|^2 - 2 c \|x^*\|  \max_{i \in [K]} \|x_{i}-x^*\| - (1+c^2)\|x^*\|^2 \leq 0
	\]
	Solving the quadratic and using $\sqrt{a + b} \leq \sqrt{a} + \sqrt{b}$ implies that
	\[
	 \max_{i \in [K]} \|x_{i}-x^*\| \leq \frac{1}{2} 
	\left[
	2 c \|x^*\| + \sqrt{4 c^2 \|x^*\|^2 + 4(1+c^2)\|x^*\|^2}
	\right]
	\leq 
	(c+(1+c\sqrt{2})) \| x^* \| ~.
	\]
	Therefore by the definition of $\delta_k$, we have
	\[\delta_k = c[c+(1+\sqrt{2}c)]\|x^*\|^2+\frac{c^2}{2}\|x^*\|^2\leq (3c^2+c)\|x^*\|^2\]
	for all $k \in [K]$ and \eqref{eq:converge_rate} follows from Theorem~\ref{thm:convergence} and that $c(1 + 3c) \leq \frac{1}{2}$ for $c \in [0, \frac{1}{4}]$.
\end{proof}

\subsection{Leveraging approximate proximal step oracle}
\label{sec:convergence_with_oracle}

Here we show how to implement and bound the convergence of Algorithm~\ref{alg:accel_framework} given an approximate proximal step oracle. First, we show that given $\lambda_{k + 1} \omega(\|y_{k+1} - \tilde{x}_k\|)$ is sufficiently close to $1$ then $y_{k + 1}$ can be computed with an approximate proximal oracle. We show that such a $y_{k + 1}$ can always be found (for suitable choice of $\sigma$) in Appendix~\ref{sec:implementation}.

\begin{lem}[Line Search Guarantee]
	\label{lem:framework_consistency}
	If in each iteration $k$ of Algorithm~\ref{alg:accel_framework} we choose $\lambda_{k + 1}$ and $y_{k+1}$ such that for $d = \|y_{k + 1} - \tilde{x}_k\|$
	\[
	\|
	\nabla g(y_{k+1}) + 
	\omega(d) (y_{k + 1} - \tilde{x}_k )
	\| 
	\leq 
	\alpha \cdot \omega(d) d+\delta  \quad
	\text{ and } \quad
	\frac{1 - \sigma}{1 - \alpha} \leq \lambda_{k+1} \omega(d) \leq 1 
	\] 
	for $\alpha \in [0, 1)$ and $\omega : \R_+ \rightarrow \R_+$ then \eqref{eq:framework} is satisfied.
\end{lem}
\begin{proof}
	Leveraging that the assumptions imply $| \lambda_{k + 1} \omega(d)  - 1| = 1 - \lambda_{k + 1} \omega(d)$ yields
	\begin{align*}
	\normFull{\lambda_{k+1}\nabla g(y_{k+1})+y_{k+1}-\tilde{x}_k }
	&\leq
	\lambda_{k + 1} \normFull{ \grad g(y_{k + 1}) + \omega(d) (y_{k + 1} - \tilde{x}_k) }
	+ \left|\lambda_{k + 1}  \omega(d) - 1\right|  \norm{y_{k + 1} - \tilde{x}_k}
	\\
	&\leq \lambda_{k + 1} \left(\alpha \cdot \omega(d) d  +\delta\right) + (1 -  \lambda_{k + 1} \omega(d)) d\\
	&=
	[1 - (1 - \alpha)  \lambda_{k + 1} \omega(d) ]d +\lambda_{k+1}\delta~.
	\end{align*}
	Since $(1 - \alpha) \lambda_{k+1}  \omega(d) \geq 1 - \sigma$ by assumption the result follows.
\end{proof}

Note that the update $x_{k+1}$ can simply be read as $x_{k+1} = x_k-a_{k+1}\cdot v_{k+1}$ where $\| v_{k+1} - \grad g(y_{k + 1})\| \leq \delta$. Consequently, $v_{k +1}$ can just be the result of a $\delta$-approximate gradient oracle (Definition~\ref{def:noisy_gradient}). Consequently, this lemma shows that Algorithm~\ref{alg:accel_framework} can be implemented with the oracles at our disposal, provided line search can be performed to achieve the guarantee of Lemma~\ref{lem:framework_consistency}. We discuss this in the next section. 

Next we bound the diameter of the iterates of the algorithm, i.e. how much the points vary. 
 
\begin{lem}[Diameter Bound] 
	\label{lem:diameter} 
	If in Algorithm~\ref{alg:framework} we have $\delta \leq  \frac{\|x^*\|}{\mu\cdot A_{K}}$ for $\mu \defeq \frac{4\sqrt{2}}{\sqrt{1-\sigma}}$ and some $K > 0$. Then for all $k \in [K]$ and $\theta \in [0,1]$ we have  $\|y_{k}-x^{*}\|\leq \mu\|x^{*}\|$ and	$\|\widetilde{x}_{\theta} - x^{*}\|\leq \mu\|x^{*}\|$ for $\widetilde{x}_{\theta}=(1-\theta)x_{k}+\theta y_{k}$.
\end{lem}

\begin{proof}
	Let $D_{k}=\|y_{k}-x^{*}\|$. Using $\widetilde{x}_{k}=\frac{A_{k}}{A_{k+1}}y_{k}+\frac{a_{k+1}}{A_{k+1}}x_{k}$,
	we have 
	\[
	\|\widetilde{x}_{k}-x^{*}\|\leq\frac{A_{k}}{A_{k+1}}D_{k}+\frac{2a_{k+1}}{A_{k+1}}\|x^{*}\|.
	\]
	Hence, $D_{k+1}\leq\frac{A_{k}}{A_{k+1}}D_{k}+\frac{2a_{k+1}}{A_{k+1}}\|x^{*}\|+\|y_{k+1}-\widetilde{x}_{k}\|$. Rescaling and summing over $k$ yields
	\begin{align*}
		D_{k+1} & \leq 2\|x^{*}\|+\|y_{k+1}-\widetilde{x}_{k}\|+\frac{A_{k}}{A_{k+1}}\|y_{k}-\widetilde{x}_{k-1}\|+\frac{A_{k-1}}{A_{k+1}}\|y_{k-1}-\widetilde{x}_{k-2}\|+\cdots\\
		& \leq2\|x^{*}\|+\frac{1}{A_{k+1}}\sum_{j=1}^{k+1}A_{j}\|y_{j}-\widetilde{x}_{j-1}\|\\
		& \leq2\|x^{*}\|+\frac{\sqrt{\sum_{j=1}^{k+1}A_{j}\lambda_{j}}}{A_{k+1}}\sqrt{\sum_{j=1}^{k+1}\frac{A_{j}}{\lambda_{j}}\|y_{j}-\widetilde{x}_{j-1}\|^{2}}\\
		& \leq2\|x^{*}\|+\frac{\sqrt{\sum_{j=1}^{k+1}\lambda_{j}}}{\sqrt{A_{k+1}}}\sqrt{\frac{2\|x^{*}\|^{2}}{1-\sigma}}\\
		& \leq2\|x^{*}\|+\frac{2\sqrt{2}}{\sqrt{1-\sigma}}\|x^*\| \leq \mu\|x^*\|
	\end{align*}
	where we used $A_{j}$ is increasing and Lemma~\ref{err:bound}
	in the third to last equation, and equation \ref{eqn:lambda} for the second to last. The assumption on the relation between $\alpha$ and $\sigma$ implies $\sigma = \frac{1+\alpha}{2} = [\frac{1}{2},1)$ and the definition of $\mu$ gives the last inequality.
	
	 The second part of the claim follows by observing that $\widetilde{x}_\theta$ is a convex combination of $x_k$ and $y_k$, therefore
	\[\|\widetilde{x}_{\theta}-x^*\|\leq\max\{\|x_{k}-x^{*}\|,\|y_{k}-x^{*}\|\} \leq \mu\|x^*\|\, .\]
\end{proof}

Finally, we bound the growth of $A_k$; this is crucial to derive the final convergence rate of the algorithm.

\begin{lem}[Growth of $A_k$]
\label{lem:growth_Ak}
Let $\rho \defeq \frac{1-\alpha}{1-\sigma} = 2$ and $\mu \defeq \frac{4\sqrt{2}}{\sqrt{1-\sigma}} = \frac{8}{\sqrt{1-\alpha}}$. If in Algorithm~\ref{alg:framework} for  $K \geq 0$ we have $\delta \leq  \frac{\|x^*\|}{\mu\cdot A_{K}}$  and $\lambda_{k} \geq \frac{1}{\rho \cdot \omega(\| y_{k} -\tilde{x}_{k-1} \|)}$ for all $k \in \{0, ..., K \}$ then for all $J \in (0, \frac{k}{2})$ we have 
\[
A_{k} \geq 
\min \left\{
\frac{4^J}{\rho \cdot \omega(\mu\|x^*\|/4)}
~ , ~ 
\frac{ (k/J)^2 }{16 \rho \cdot \omega\left(\frac{4\mu\|x^*\|}{(k/J)^{3/2}} \right) }
\right\}
~.
\]
Further, if $\|x^*\|\leq R$ then $A_{k}\geq\frac{1}{2\omega(2\mu R)}$ for all $k \in [K]$.
\end{lem}
\begin{proof}
Let $d_k \defeq \|y_{k } - \tilde{x}_{k-1}\|$. By \eqref{eq:converge_rate} of Lemma~\ref{err:bound} we obtain for all $k \in [K]$ 
\begin{equation}
\label{eq:bound_by_inverse_formula}
\sum_{i \in [k]} \frac{A_i}{\lambda_i} d_i^2 \leq \frac{2\|x^{*}\|^2}{1-\sigma} .
\end{equation}
Since $A_0 = 0$ we have $A_ 1 = a_1 = \lambda_1$ and consequently, \eqref{eq:bound_by_inverse_formula} yields $d_1^2 \leq \frac{2\|x^{*}\|^{2}}{1-\sigma} $ and therefore $d_1 \leq\frac{\mu}{4}\|x^*\|$. Since $\omega$ is monotonic the assumptions imply 
\[A_1 = \lambda_1 \geq \frac{1}{\rho \cdot \omega(d_1)} \geq \frac{1}{\rho \cdot \omega(\mu\|x^*\|/4)}.\] 
Since the $A_k$ increase monotonically this immediately implies $A_k \geq A_1 \geq 1/[\rho\cdot \omega(\mu \|x^*\| / 4)]$ as desired. Further, this implies that if $A_{k} \geq 4^J A_1$ then the first term in the result holds. 

On the other hand, suppose $A_{k} < 4^J A_1$. Then, for some $1 \leq i \leq j \leq k$ we have $A_{j} < 4 A_{i}$ and $|j - i|\geq k/J$. The construction of $A_k$ then implies
\begin{equation}
\label{eqn:lambda}
\sqrt{A_j}
>
\sqrt{A_j} - \sqrt{A_i}
= \sum_{t = i}^{j - 1} \left[\sqrt{A_{t +1}} - \sqrt{A_{t}}  \right]
= \sum_{t = i}^{j - 1} \frac{a_{t + 1}}{\sqrt{A_{t + 1}}+\sqrt{A_{t}}} 
\geq 
\frac{1}{2} 
\sum_{t = i}^{j - 1}
\sqrt{\lambda_{t + 1}}
\end{equation}
Hence, at least $\left\lceil \frac{j-i}{2}\right\rceil$ many $\lambda$'s
have value less than $\frac{16 A_{j}}{(j-i)^{2}}$. Letting $S$ denote the indices of these $\lambda$ we have by \eqref{eq:bound_by_inverse_formula} that
\[
\frac{2\|x^{*}\|^2}{1-\sigma}  \geq \sum_{t \in S} \frac{A_t}{\lambda_t} d_t^2
\geq 
\left\lceil \frac{j-i}{2}\right\rceil \frac{A_{i}}{\left(\frac{16A_{j}}{(j-i)^{2}}\right)} \cdot \frac{1}{|S|} \sum_{t \in S} d_t^2
\geq 
\frac{(k/J)^3}{32\cdot 4} \cdot \frac{1}{|S|} \sum_{t \in S} d_t^2
\]
Consequently, $d_t \leq \frac{16}{\sqrt{1-\sigma}}\frac{ \|x^*\|}{(k/J)^{3/2}}\leq \frac{4\mu\|x^*\|}{(k/J)^{3/2}}$ and $\lambda_t < \frac{16 A_j}{(j - i)^2} \leq \frac{16 A_j}{(k/J)^2}$ for some $t \in [k]$. However, the monotonicity of $\omega$ and the assumptions on $\lambda$ also imply 
\[\lambda_t \geq \frac{1}{\rho \cdot \omega(d_t)} \geq \frac{1}{\rho \cdot \omega(\frac{4\mu\|x^*\|}{(k/J)^{3/2}})}\]
and the result now follows by observing that
\[A_{k} \geq A_t \geq \lambda_t \frac{(k/J)^2}{16}\]
giving the second term in the result.
\end{proof}

\subsection{Putting it all together}
\label{sec:put-together}

Here we put together the analysis from the preceding sections and prove Theorem~\ref{thm:main_acceleration}. Our proof  relies on the following  theorem giving our main guarantee regarding such a line search algorithm (See Section~\ref{sec:implementation} for the proof.) 

\begin{restatable}[Line Search Algorithm]{theorem}{lsrestate}
	\label{thm:main_line_search}
	Let $g : \R^d \rightarrow \R$ be a twice differentiable function that is minimized at a point $x^* \in \R^d$ with $\|x^*\| \leq R$.  Further, let $\omega: \mathbb{R}_+ \rightarrow \mathbb{R}_+$ be a continuously differentiable function where $0<\omega'(s)\leq\gamma\frac{\omega(s)}{s}$ for some fixed $\gamma\geq1$ and all $s>0$. Further, let $\mu \defeq \frac{8}{\sqrt{1-\alpha}}$ and suppose 
	\[
	\delta \leq \min\left\{\frac{\epsilon}{\mu\cdot R\cdot  9c[(1+\alpha)c+1]} ~ , ~ 8\mu R\cdot \omega(8\mu R)\right\}
	\text{ and }
	64\left(\alpha+\frac{1}{c}\right)\gamma^2\leq1 \text{ for some } 
	c \geq 1 ~.
	\]
	Then for any inputs $x^{(1)},x^{(2)}$ with $\|x^{(1)}-x^*\|,\|x^{(2)}-x^*\| \leq \mu R$, $\frac{1}{2\omega(2\mu R)}\leq A\leq\frac{R^2}{\epsilon}$ there is an algorithm that returns $y$ and $\lambda$ such that $\widetilde{x}=\frac{a}{A+a}x^{(1)}+\frac{A}{A+a}x^{(2)}$ for $a=\frac{\lambda+\sqrt{\lambda^{2}+4\lambda A}}{2}$ that either satisfies 
	\[
	g(y)\leq g^*+\epsilon \quad \text{and} \quad \omega(\|y-\widetilde{x}\|)\cdot\|y-\widetilde{x}\|\leq c\cdot \delta
	\]
	or, satisfies
	\[
	\frac{1}{2} \leq\lambda\cdot\omega(\|y-\widetilde{x}\|)\leq 1
	\quad\text{, }\quad \omega(\|y-\widetilde{x}\|)\cdot\|y-\widetilde{x}\| > c\cdot \delta
	\text{, }
	\]
	and 
	\[
	\left\Vert \nabla g(y)+\omega(\|y-\widetilde{x}\|)\cdot(y-\widetilde{x})\right\Vert \leq \alpha \cdot\omega(\|y-\widetilde{x}\|)\cdot\|y-\widetilde{x}\|+\delta
	\]
	after 
	\[6+\log_{2}\Big[\Big(\frac{160\mu Rc}{\delta}+\frac{9R^2}{\epsilon}\Big)\cdot \omega(8c\mu R)\Big]\]
	calls to the $(\alpha, \delta)$-approximate $\omega$-proximal step oracle $\oracleprox$ for $g$.
\end{restatable}

Leveraging this we can prove our main theorem regarding our acceleration framework. We first give this result below as a slightly more general result and then use it to immediately improve the theorem.

\begin{theorem}[General Tunable Acceleration Framework]
	\label{thm:main_acceleration_tuanbale}
	Let $g : \R^d \rightarrow \R$ be a convex twice-differentiable function minimized at $x^*$ with $\|x^*\| \leq R$, $\epsilon > 0$, $\alpha \in [0, 1)$, and $c\geq 150, \gamma \geq 1$ such that $64 (\alpha + c^{-1}) \gamma^2 \leq 1$. Further, let $\omega : \R_+ \rightarrow \R_+$ be a monotonically increasing continuously differentiable function with $0 < \omega'(s) \leq \gamma\cdot \omega(s) / s$ for all $s > 0$.  There is an algorithm which for all $k$ computes a point $y_k$ with
	\[
	g(y_{k})-g^* \leq 
	\max \Bigg\{\epsilon ~ , ~
	\frac{32\cdot\omega\left(\frac{4\mu\|x^{*}\|}{k^{3/2}}\right)\|x^{*}\|^2}{k^{2}}\Bigg\}
	\quad \text{where} \quad \mu \defeq \frac{8}{\sqrt{1-\alpha}} 
	\]
	using  $k (6+\log_2[(1500 \mu^2 R^2 c^2 [(1 + \alpha) c + 1]) \cdot \omega(8c\mu R) \cdot \epsilon^{-1}])^2$  queries to a $(\alpha, \delta)$-approximate $\omega$-proximal step oracle for $g$ and a $\delta$-approximate gradient oracle for $g$ provided  that it holds that
	$\delta \leq \epsilon/[20 \mu^3 R [(1 + \alpha)c + 1]]$ and $\epsilon \leq  72 c[(1+\alpha)c+1] (\mu R)^3 \cdot \omega(8\mu R)$.
	
\end{theorem}

\begin{proof} Consider an application of 
Algorithm~\ref{alg:accel_framework} where in each iteration $k$ we invoke Theorem~\ref{thm:main_line_search} with $x^{(1)} = y_k$, $x^{(2)} = x_k$, and $A = A_k$ to compute $y_{k + 1} = y$ and $\lambda_k = \lambda$. Now supposing that $A_k \leq R^2 / \epsilon$ and that in this invocation we choose the $\delta$ of Theorem~\ref{thm:main_line_search} to be $\delta' \defeq \min\{ \epsilon' / (\mu R) , 8 \mu R \cdot \omega(8\mu R) \} = \epsilon' / (\mu R)$ for $\epsilon' \defeq \epsilon / [9c[(1+\alpha)c+1]]$ , we have that the conditions of Lemma~\ref{err:bound} and Theorem~\ref{thm:main_line_search} are met as $\epsilon' \leq \epsilon$.
Further, if $\omega(\|y-\widetilde{x}\|_{2})\cdot\|y-\widetilde{x}\|_{2} \leq c\cdot \delta'$ then we output $y_{k + 1}$ and are guaranteed that $g(y_{k + 1}) \leq g^* + \epsilon$ by Theorem~\ref{thm:main_line_search} and the choice of parameters. 
 
 Otherwise, $\omega(\|y-\widetilde{x}\|_{2})\cdot\|y-\widetilde{x}\|_{2} > c\cdot \delta'$ and the necessary conditions are met for Algorithm~\ref{alg:accel_framework} to proceed by Lemma~\ref{lem:framework_consistency}. Further, in this case, we have that  
 \[
 \lambda_{k + 1} \leq  \frac{1}{\omega(\|y_{k + 1} -\tilde{x}_{k}\|)} 
 \leq \frac{\|y_{k + 1} -\tilde{x}_{k}\|}{c\cdot \delta'} \leq \frac{2\mu \|x^*\|}{c\cdot\delta'}\] 
 where we used Lemma~\ref{lem:diameter} for the last inequality. Furthermore, the assumption that $A_{k} \leq \frac{\|x^*\|^2}{\epsilon}$, Remark~\ref{rem:a_k}, and the assumption on $\delta$ yield (using $ab \leq \frac{1}{2}a^2+\frac{1}{2}b^2  $)
 \begin{align*}
 A_{k + 1} &= A_{k} + a_{k + 1} \leq A_{k}  + \sqrt {A_{k + 1}}\cdot \sqrt{\frac{2\mu \|x^*\|}{c\cdot\delta'}}
 \leq \frac{\|x^*\|^2}{\epsilon}+
 \frac{1}{2} A_{k + 1} + \frac{\mu \|x^*\|}{c \delta'} 
 \end{align*}
which implies that 
\[
A_{k+1} \leq \frac{2\|x^*\|^2}{\epsilon}+\frac{2\mu\|x^*\|}{c\cdot\delta'}
\leq 
 \frac{2 R^2}{\epsilon}+\frac{18 \mu^2 R^2 [(1 + \alpha)c + 1]}{\epsilon}
\leq \frac{20 \mu^2 R^2 [(1 + \alpha)c + 1]}{\epsilon}
\] 
Since, $\|x^*\|/ (\mu A_{k + 1}) \geq \epsilon/[20 \mu^3 R [(1 + \alpha)c + 1]] \geq \delta$ by the assumption $c \geq 150$, we have that Lemma~\ref{err:bound} still holds and therefore either $A_{k + 1}\leq \|x^*\|^2 / \epsilon$ or $g(y_{k+1}) - g^* \leq \epsilon$ and we can repeat the inductive argument. 
 
Consequently, if after $k$ steps we have not already returned an $\epsilon$-approximate point then we have from Lemma~\ref{err:bound} and Lemma~\ref{lem:growth_Ak} the convergence rate to an $\epsilon$-optimal point of the general framework as
\begin{align*}
g(y_{k})-g^* &\leq \frac{\|x^*\|^2}{A_{k}} \leq \min_{J\in[\frac{k}{2}]} \max\left\{\frac{2\cdot \omega(\mu\|x^*\|/4)}{4^{J}},\frac{32\cdot\omega\left(\frac{4\mu\|x^{*}\|}{(k/J)^{3/2}}\right)}{(k/J)^{2}}\right\}\|x^{*}\|^2 
\end{align*}
and the convergence rate follows by considering $J = \lceil1 + \log_4( 2\|x^*\|^2 \omega(\mu \|x^*\|/4)/\epsilon)\rceil$ and the monotonicity of $\omega$. Putting together with Theorem~\ref{thm:main_line_search}, we have that for
\begin{align*}
\mathcal{K} &\defeq \Big\lceil1 + \log_4\Big( \frac{2\|x^*\|^2 \omega(\mu \|x^*\|/4)}{\epsilon}\Big)\Big\rceil \cdot \Big(6+\log_{2}\Big[\Big(\frac{160\mu \|x^*\| c}{\delta'}+\frac{9\|x^*\|^2}{\epsilon}\Big)\cdot \omega(8c\mu \|x^*\|)\Big]\Big) \\
&\leq \Big(6+\log_2\Big[\frac{170\mu^2R^2c}{\epsilon'}\cdot\omega(8c\mu R)\Big]\Big)\cdot\Big\lceil1 +\frac{1}{2}\log_2\Big(2R^2\frac{\omega(\frac{\mu R}{4})}{\epsilon}\Big) \Big\rceil \\
&\leq \Big(6+\log_2\big[\frac{170\mu^2R^2c}{\epsilon'}\cdot\omega(8c\mu R)\big]\Big)^2
 \leq  \Big(6+\log_2\Big[\frac{1500 \mu^2 R^2 c^2 [(1 + \alpha) c + 1]}{\epsilon}\cdot\omega(8c\mu R)\Big]\Big)^2
\end{align*}
$\mathcal{K}$ queries to a $(\alpha, \delta')$-approximate $\omega$-proximal step oracle is needed at each iteration.

\end{proof}

Leveraging this, we prove Theorem~\ref{thm:main_acceleration}.

\begin{proof}[Proof of Theorem~\ref{thm:main_acceleration}]
	Consider invoking Theorem~\ref{thm:main_acceleration_tuanbale} with $c = 150 \gamma^2$. Since $\gamma \geq 1$ we have $c \geq 150$. Further, since $\alpha \leq 1/(128 \gamma^2)$ and $c^{-1} \leq 1/(128 \gamma^2)$ we have $64(\alpha + c^{-1}) \gamma^2 \leq 1$. Further, under these assumptions we have $\mu  \defeq 8/(\sqrt{1- \alpha}) \leq 10$ and $[(1 + \alpha)c  + 1] \leq 200 \gamma^2$. Consequently, $\delta$ and $\epsilon$ are constrained sufficiently to invoke Theorem~\ref{thm:main_acceleration_tuanbale} and the result follows.
\end{proof}

\section{Applications}
\label{sec:applications}

Here we briefly sketch several applications of the acceleration framework described in Section~\ref{sec:acceleration_framework}. First we show how minimizing the regularized $p$-th order Taylor approximation to $g$ is  yields an approximate $\omega$-proximal step oracle.

\begin{lem}[Accelerated Taylor Descent]
Suppose that $\nabla^p g$ is $L_p$-Lipschitz and that $\oracle(x) \defeq \argmin_{y}\; g_p(y;x) + \frac{L_p + L}{p!} \|y - x\|^{p + 1}$ where 
	$g_p(y;x)$ is the value of the $p$'th order Taylor approximation of $g$ about $x$ evaluated at $y$ and $L \geq 0$. Then, $\oracleprox$ is a $((1 + p)^{-1} (1 + L/L_p)^{-1},0)$-approximate $\omega$-proximal step oracle (Definition~\ref{def:prox_oracle}) for $\omega(d) \defeq \frac{(L_p + L) \cdot (p + 1)}{p!} d^{p -1}$.
\end{lem}

\begin{proof}
	Let $y = \oracleprox(x)$ for arbitrary $x$. The optimality conditions of $y$ yield that 
	\[
	\grad_y g_p(y;x) = \frac{(p + 1)(L_p + L)}{p!} \norm{y -x}^{p - 1} (x -y) =  \omega (\norm{y- x}) (x - y) ~.
	\]
	Further, since Taylor expansion of $\grad g(y)$ yields
\begin{align*}
\|\nabla g(y)+ \omega(\|y - x\|) (y - x) \|
&= \|\grad g(y) - \grad_y g_p(y;x)\| \leq \frac{L_p}{p!} \norm{y - x}^p \\
&= \frac{L_p}{(1 + p)(L_p + L)} \omega(\| y - x\|) \|y - x\|
\end{align*}
	the result follows by observing that $\alpha =  (1 + p)^{-1} (1 + L/L_p)^{-1}$ and $\delta = 0$, as claimed.
\end{proof}

Now, note that for $\omega(d)$ defined in this lemma we have that $\omega'(d) = (p - 2) \omega(s)/s$. Consequently, with respect to Theorem~\ref{thm:main_acceleration} we have that $\gamma = p - 2$ and $\alpha = (1 + p)^{-1} (1 + L/L_p)^{-1}$ for the oracle defined in this lemma. Consequently, by picking $L = \mathcal{O}(L_p \cdot \mathrm{poly}(p))$ this oracle satisfies the necessary conditions of the theorems and therefore (up to logarithmic factors) with $k$ queries to the oracle and a gradient oracle invoking Theorem~\ref{thm:main_acceleration} yields that one can compute a point $y_k$ with 
\[g(y_{k})-g^* \lesssim \frac{\omega(\frac{\|x^*\|}{k^{3/2}})\|x^*\|^2}{k^2} \lesssim \frac{(L_p + L)\cdot (p+1)\cdot\|x^*\|^{p+1}}{p! \cdot k^\frac{3p+1}{2}}\, .\]
This matches the rate of  \citep{Gasnikov18, JWZ18, BJLLS8}  up to polylogarithmic factors. 

Next we show how approximately minimizing a regularization of $g$ yields an approximate $\omega$-proximal step oracle.

\begin{lem}[Approximate Proximal Point]
Suppose that $g$ is $L$-smooth and convex and that $\oracle(x)$ is a point $y_x$ where for $G_x(y) \defeq g(y) + \frac{\kappa}{2} \|y - x\|^2$ we have $G_x(y_x) - G_x^* \leq \rho$ where $G_x^*$ is the minimum value of $G_x$. Then, $\oracleprox$ is a $(0,\rho (L + \kappa))$-approximate $\omega$-proximal step oracle (Definition~\ref{def:prox_oracle}) for $\omega(d) \defeq \kappa$.
\end{lem}

\begin{proof}
Since $G$ is $L+\kappa$-smooth we have that
\[
\rho \geq \frac{1}{L + \kappa} \norm{\grad G_x(y_x)} = \frac{1}{L + \kappa} \|\grad g(y_x) + \kappa(y_x - x)\| ~.
\]	
The result follows by observing that $\alpha = 0$ and $\delta = \rho (L + \kappa)$, as claimed.
\end{proof}

Now, note that for $\omega(d)$ defined in this lemma we have that $\omega'(d) = 0$. Consequently, with respect to Theorem~\ref{thm:main_acceleration} we have that $\gamma =0$ and $\alpha = 0$ for the oracle defined in this lemma. Consequently, this oracle satisfies the necessary conditions of the theorems for some $\epsilon$ so long as $\rho = \mathcal{O}(\epsilon/[\|x^*\| (L + \kappa)])$ and therefore (up to logarithmic factors) with $k$ queries to the oracle and a gradient oracle invoking Theorem~\ref{thm:main_acceleration} yields that one can compute a point $y_k$ with 
\[g(y_{k})-g^* \lesssim \frac{\omega(\frac{\|x^*\|}{k^{3/2}})\|x^*\|^2}{k^2} \lesssim \frac{\kappa \cdot\|x^*\|^{2}}{k^2}\, .\]
This matches the rate of  \citep{FrostigGKS15, LinMH15}  up to polylogarithmic factors with slightly stronger assumptions. We leave it to future work to use this framework to fully generalize this result and develop further applications.

\section{Upper bound}
\label{sec:parallel_upper}

Here we provide the proofs associated with Section~\ref{sec:highly_parallel_optimization} and prove Theorem~\ref{thm:parallel_min_main}. Our proof is split into several parts. In Section~\ref{sec:convolution} we provide basic facts about the convolved function we optimize, 
in Section~\ref{sec:noisy_gradient} we analyze our algorithm for approximating the gradient, in Section~\ref{sec:approx_step_implement} we analyze our algorithm for computing an approximate proximal step, and in Section~\ref{sec:parallel_complexity} we put everything together to prove Theorem~\ref{thm:parallel_min_main}. 

Throughout this section we use $\|\cdot\|_{op}$ to denote the operator norm of a matrix and $D$ as the differential operator.

\subsection{Gaussian convolution for approximation}
\label{sec:convolution}

Here we prove Lemma~\ref{lem:convolution_err} which provides basic facts about $g$, e.g. convexity and continuity, that we use throughout our analysis.

\begin{proof}[Proof of Lemma~\ref{lem:convolution_err}]
Since $g$ is a weighted linear combination of shifted $f$, i.e.
\[g(y) = \int_{\mathbb{R}^d}\gamma_{r}(x)f(y-x)dx\]
and as $f$ is convex, so is $g$. Similarly, we have $g$ is $L$-Lipschitz.
Finally, we note that
\[
|g(y)-f(y)|\leq\int_{\R^d}\gamma_{r}(y-x)|f(x)-f(y)|dx 
\leq L \int_{\mathbb{R}^d}\gamma_{r}(y-x)\|x-y\|_{2}dx
=L \cdot \E_{x\sim\gamma_{r}}\|x\|_{2}\leq L\sqrt{d}\cdot r
\]
where we used $\E_{x\sim\gamma_{r}}\|x\|_{2}\leq\sqrt{\E_{x\sim\gamma_{r}}\|x\|_{2}^{2}}\leq\sqrt{d}\cdot r$.

Next, we note that $\nabla g=\gamma_{r}\ast\nabla f$ and hence $\nabla^2 g = \nabla \gamma_r\ast \nabla f$
\[
v^{\top}\nabla^{2}g(y)v=\int_{\mathbb{R}^d}\gamma_{r}(y-x)\cdot\left\langle -\frac{y-x}{r^{2}},v\right\rangle \cdot\left\langle \nabla f(x),v\right\rangle dy.
\]
So we have for any $\|v\|_2 = 1$, by the fact that $f$ is $L$-Lipschitz that 
\begin{align*}
\left|v^{\top}\nabla^{2}g(y)v\right|  \leq\frac{L}{r}\cdot\int_{\mathbb{R}^d}\gamma_{r}(y-x)\left|\left\langle \frac{y-x}{r},v\right\rangle \right|dy
 =\frac{L}{r}\cdot\E_{\zeta\sim \mathcal{N}(0,1)}|\zeta|
 =\frac{L}{r}\cdot \sqrt{\frac{2}{\pi}}\leq \frac{L}{r}. 
\end{align*}
and therefore $\|\nabla^{2}g(y)\|_{\op}\leq\frac{L}{r}$.
\end{proof}

\subsection{Noisy gradient oracle: sampling}
\label{sec:noisy_gradient}

In this section we prove Lemma~\ref{lem:uniform-approximation} bounding the performance of Algorithm~\ref{alg:apx} for approximating the gradient of $g$. We begin by studying each sampled vector in Algorithm \ref{alg:apx}.

\begin{lem}[Statistics of one sample]
\label{lem:mean_bound}Given a $L$-Lipschitz function
$f$ on $\mathbb{R}^{d}$, a vector $c$, radius $r>0$, and error
parameter $1>\eta>0$. Sample $x$ according to $\gamma{}_{r}(x-c)$.
Define the vector field 
\[
\ell(y)\defeq\frac{\gamma_{r}(y-x)}{\gamma{}_{r}(c-x)}\cdot\nabla f(x)\cdot\chi((x-c)^{\top}(y-c))\cdot1_{\|x-c\|\leq(\sqrt{d}+\frac{1}{\eta})r}.
\]
For any $y$ such that $\|y-c\|\leq\frac{\eta}{4}r$, we have that
\begin{align*}
\|\E\ell(y)-\nabla(\gamma_{r}\ast f)(y)\|_{2} & \leq2L\cdot\exp(-\frac{1}{2\eta^{2}}),\\
\|\ell(y)\|_{2} & \leq3L,\\
\|D\ell(y)\|_{\op} & \leq\frac{20L\sqrt{d}}{r\eta}.
\end{align*}
\end{lem}

\begin{proof} For the bias, we note that 
\begin{align*}
\E\ell(y) & =\int_{\R^{d}}\frac{\gamma_{r}(y-x)}{\gamma{}_{r}(x-c)}\cdot\nabla f(x)\cdot\chi((x-c)^{\top}(y-c))\cdot\gamma{}_{r}(x-c)\cdot1_{\|x-c\|\leq(\sqrt{d}+\frac{1}{\eta})r}\,dx\\
 & =\int_{\R^{d}}\gamma_{r}(y-x)\cdot\nabla f(x)\cdot\chi((x-c)^{\top}(y-c))\cdot1_{\|x-c\|\leq(\sqrt{d}+\frac{1}{\eta})r}\,dx\\
 & =\nabla(\gamma_{r}\ast f)(y)-\int_{\R^{d}}\gamma_{r}(y-x)\cdot\nabla f(x)\cdot\beta(y,x)\,dx
\end{align*}
where
\[
\beta(y,x)=1-\chi((x-c)^{\top}(y-c))\cdot1_{\|x-c\|\leq(\sqrt{d}+\frac{1}{\eta})r}.
\]
Since $1\geq\beta(y,x)\geq0$ for all $x,y$, we have
\begin{align*}
\|\E\ell(y)-\nabla(\gamma_{r}\ast f)(y)\|_{2} & \leq\int_{\R^{d}}\gamma_{r}(y-x)\cdot\|\nabla f(x)\|_{2}\cdot\beta(y,x)\,dx\\
 & \leq L\cdot\int_{\R^{d}}\gamma_{r}(y-x)\cdot\beta(y,x)\,dx\\
 & \leq L\cdot\Prob_{x}\left[\beta(y,x)>0\right].
\end{align*}

Now, we note that $\beta(y,x)>0$ implies either $\|x-c\|>(\sqrt{d}+\frac{1}{\eta})r$
or $|(x-c)^{\top}(y-c)|>\frac{r^{2}}{2}$.

By a tail bound of Chi-square distribution \citep{laurent2000}, we have
\begin{equation}
\Prob_{x}\left(\|x-c\|>\left(\sqrt{d}+\frac{1}{\eta}\right)r\right)\leq\exp\left(-\frac{1}{2\eta^{2}}\right).\label{eq:mean_bound_case_1}
\end{equation}

Next, we note that for any fixed $c$ and $y$, $(x-c)^{\top}(y-c)$
follows the normal distribution $\mathcal{N}(\|y-c\|^{2},\|y-c\|^{2}r^{2})$
when $x$ is sampled from $\gamma_{r}(y-x)$. By the assumption that
$\|y-c\|\leq\frac{\eta}{4}r \leq \frac{r}{4}$, we have that 
\begin{align}
\Prob_{x}\left[|(x-c)^{\top}(y-c)|>\frac{r^{2}}{2}\right] & \leq\Prob_{\zeta\sim\mathcal{N}(0,1)}\Big(|\zeta|\geq\frac{r^{2}}{4\|y-c\|r}\Big)\nonumber  \leq\exp\Big(-\frac{r^{2}}{32\|y-c\|^{2}}\Big)\nonumber \\
 & \leq\exp\left(-\frac{1}{2\eta^{2}} \right).\label{eq:mean_bound_case_2}
\end{align}

Union bound over case \eqref{eq:mean_bound_case_1} and case \eqref{eq:mean_bound_case_2}
gives that $\Prob_{x}\left[\beta(y,x)>0\right]\leq2\exp(-\frac{1}{2\eta^{2}}).$
This gives the bound on $\E\ell(y)$.

For the bound on $\|\ell\|$, we note from the Lipschitz assumption
of $f$ that 
\begin{align*}
\|\ell(y)\|_{2} & \leq L\cdot\frac{\gamma_{r}(y-x)}{\gamma{}_{r}(x-c)}\cdot\chi((x-c)^{\top}(y-c)).\\
 & \leq L\cdot\frac{\gamma_{r}(y-x)}{\gamma{}_{r}(x-c)}\cdot1_{|(x-c)^{\top}(y-c)|\leq r^{2}}
\end{align*}
For any $x$ with $|(x-c)^{\top}(y-c)|\leq r^{2}$, we have that
\begin{align}
\log\frac{\gamma_{r}(y-x)}{\gamma{}_{r}(x-c)} & =-\frac{1}{2r^{2}}\|y-x\|_{2}^{2}+\frac{1}{2r^{2}}\|c-x\|_{2}^{2}\nonumber \\
 & =\frac{1}{2r^{2}}\left(-2(c-x)^{\top}(y-c)-\|y-c\|_{2}^{2}\right)\nonumber \\
 & \leq\frac{|(x-c)^{\top}(y-c)|}{r^{2}}<1.\label{eq:change_N}
\end{align}
Hence, we have $\|\ell(y)\|_{2}\leq3L$.

For the bound of the Jacobian of $\ell$, we note that 
\begin{align*}
D\ell(y)= & \frac{\gamma_{r}(y-x)}{\gamma{}_{r}(x-c)}\cdot\nabla f(x)\cdot\Big(-\frac{y-x}{r^{2}}\Big)^{\top}\cdot\chi((x-c)^{\top}(y-c))\cdot1_{\|x-c\|\leq(\sqrt{d}+\frac{1}{\eta})r}\\
 & +\frac{\gamma_{r}(y-x)}{\gamma{}_{r}(c-x)}\cdot\nabla f(x)\cdot(x-c)^{\top}\chi'((x-c)^{\top}(y-c))\cdot1_{\|x-c\|\leq(\sqrt{d}+\frac{1}{\eta})r}.
\end{align*}
Since the Lipschitz constant of $\chi$ is bounded by $\frac{2}{r^{2}}$
and the Lipschitz assumption of $f$ is bounded by $L$, \eqref{eq:change_N}
and the above equation shows that
\begin{align*}
\|D\ell(y)\|_{\op} & \leq e\cdot L\cdot\frac{\|y-x\|_{2}}{r^{2}}\cdot1_{\|x-c\|\leq(\sqrt{d}+\frac{1}{\eta})r}+e\cdot L\cdot\|x-c\|_{2}\cdot\frac{2}{r^{2}}\cdot1_{\|x-c\|\leq(\sqrt{d}+\frac{1}{\eta})r}\\
&\leq e\cdot L\cdot \frac{(\sqrt{d}+\frac{1}{\eta}+\frac{\eta}{4})r}{r^2}+2e\cdot L\frac{(\sqrt{d}+\frac{1}{\eta})r}{r^2}\\
& \leq\frac{Lr}{r^{2}}+\frac{9L}{r}\cdot \left(\sqrt{d}+\frac{1}{\eta} \right)\leq\frac{20L\sqrt{d}}{r\eta}.
\end{align*}
\end{proof}

By a concentration and $\epsilon$-net argument we use Lemma~\ref{lem:mean_bound} to prove Lemma~\ref{lem:uniform-approximation}.

\begin{proof}[Proof of Lemma~\ref{lem:uniform-approximation}] Fix $y$ such that $\|y-c\|_{2}\leq\frac{\eta}{4}r$ and let $v(y)-\nabla g(y)=\frac{1}{N}\sum_{i=1}^{N}\epsilon^{(i)}$
be the sum of $N$ independent vectors $\epsilon^{(i)}$ corresponding to the sample $\ell$ of Lemma \ref{lem:mean_bound}. Lemma~\ref{lem:mean_bound}
shows that 
\[
\|\E\epsilon^{(i)}\|_{2}\leq2L\cdot\exp\Big(-\frac{1}{2\eta^{2}}\Big)\text{ and }\|\epsilon^{(i)}\|_{2}\leq3L+L=4L.
\]
Pinelis's inequality \citep{pinelis1994} shows that
\[
\Prob\left(\left\|\frac{1}{N}\sum_{i=1}^{N}\epsilon^{(i)}\right\|_{2}\geq2L\cdot\exp \left(-\frac{1}{2\eta^{2}} \right)+4L\cdot t\right)\leq2\exp\left(-\frac{Nt^{2}}{2}\right).
\]

To make this holds for all $y$ with $\|y-c\|_{2}\leq\frac{\eta}{4}r$,
we pick an $\epsilon$-net $\mathcal{N}$ on $\{y:\|y-c\|_{2}\leq\frac{\eta}{4}r\}$
with $\epsilon=\frac{\eta}{4}r\cdot\frac{\exp(-\frac{1}{2\eta^{2}})}{3\sqrt{d}}$.
It is known that $|\mathcal{N}_{\epsilon}(B_d(0,r))|\leq(\frac{3r}{\epsilon})^{d}$, therefore, using $0<\eta\leq1$
\[
|\mathcal{N}|\leq\left(\frac{3\frac{\eta}{4}r}{\frac{\eta}{4}r\cdot\frac{\exp(-\frac{1}{2\eta^{2}})}{3\sqrt{d}}}\right)^{d}=\left(9\sqrt{d}\right)^{d}\exp\left(\frac{d}{2\eta^{2}}\right)\leq\exp\left(\frac{d\log(81d)}{\eta^{2}}\right).
\]
For any $y$ with $\|y-c\|_{2}\leq\frac{\eta}{4}r$, there is $y'\in\mathcal{N}$
with $\|y'-y\|_{2}\leq\epsilon$, therefore by Lemma \ref{lem:convolution_err}
we have 
\[
\|\nabla g(y')-\nabla g(y)\|\leq\frac{L\epsilon}{r}\leq L\cdot\exp \left(-\frac{1}{2\eta^{2}} \right).
\]
Lemma \ref{lem:mean_bound} shows that $\|Dv(y)\|_{\op}\leq\frac{20L\sqrt{d}}{r\eta}$. Hence, we have 
\[
\|v(y')-v(y)\|_{2}\leq\epsilon\cdot\frac{20L\sqrt{d}}{r\eta}\leq2L\exp\left(-\frac{1}{2\eta^{2}}\right).
\]
Taking the union bound on $\mathcal{N}$, we have that
\[
\Prob\left(\max_{y:\|y-c\|_{2}\leq \frac{\eta}{4}r}\|v(y)-\nabla g(y)\|_{2}\geq5L\cdot\exp\left(-\frac{1}{2\eta^{2}}\right)+4L\cdot t\right)\leq2\exp\left(\frac{d\log(81d)}{\eta^{2}}-\frac{Nt^{2}}{2}\right).
\]
Setting $2\exp\left(\frac{d\log(81d)}{\eta^{2}}-\frac{Nt^{2}}{2}\right)=\delta$,
we get
\[
4L\cdot t\leq\frac{4L}{\sqrt{N}}\sqrt{\frac{2d\log(81d)}{\eta^{2}}+2\log\frac{2}{\delta}}\leq\frac{8L}{\sqrt{N}}\sqrt{\frac{d\log(9d)}{\eta^{2}}+\log\frac{1}{\delta}}
\]
on the LHS. 
\end{proof}

\subsection{Approximate proximal step oracle implementation}
\label{sec:approx_step_implement}

Here we prove the following theorem which bounds the performance of Algorithm~\ref{alg:apx-2}.

\begin{theorem} \label{thm:optimization_oracle} Algorithm \ref{alg:apx-2}
	outputs $y$ such that 
	$
	\| \nabla g(y)+\omega(\|y-c\|)\cdot(y-c) \| \leq L\cdot\epsilon
	$
	in $\mathcal{O}(\frac{p\sqrt{d}}{\epsilon^{2}})$ iterations with
	$N=\mathcal{O}([d\log d\log(\frac{1}{\epsilon})+\log(\frac{1}{\delta})]\epsilon^{-2})$
	oracle calls to $f$ in parallel with probability at least $1-\delta$ where $\omega$ is defined by (\ref{eq:Lip_omega}) with $\tilde{r}=\frac{r}{8\sqrt{\log\left(\frac{60}{\epsilon}\right)}}$.
\end{theorem}

\begin{proof}[Proof of Theorem~\ref{thm:optimization_oracle}] First, we need to prove that $y$ stays inside $\|y-c\|_{2}\leq\tilde{r}$.
Given this, the correctness of the output follows from the error bound
on $v$ and the stopping condition.

We prove $\|y-c\|_{2}\leq\tilde{r}$ by induction. Let $y'$ be the
one step from $y$, namely $y'=y-h\cdot\delta_{y}$. Then, we have
\begin{align*}
\|y'-c\|_{2} & \leq\|y-h\cdot\omega(\|y-c\|_{2})\cdot(y-c)-c\|_{2}+h\|v(y)\|_{2}\\
 & =\left|1-h\cdot\omega(\|y-c\|_{2})\right|\|y-c\|_{2}+\frac{4}{3}Lh
\end{align*}
where we used the induction hypothesis $\|y-c\|_{2}\leq\tilde{r}$
and the approximation guarantee to show that $\|v(y)\|_{2}\leq\|v(y)-\nabla g(y)\|_{2}+\|\nabla g(y)\|_{2}\leq\frac{L\epsilon}{6}+L\leq\frac{4}{3}L$.
Next, we note from the assumption on step size that 
\[
h\cdot\omega(\|y-c\|_{2})\leq h\cdot\frac{4L\tilde{r}^{p}}{\tilde{r}^{p+1}}\leq1.
\]
Hence, we have 
\begin{align*}
\|y'-c\|_{2} & \leq\left(1-h\cdot\omega(\|y-c\|_{2})\right)\|y-c\|_{2}+\frac{4}{3}Lh\\
 & =\|y-c\|_{2}-h\omega(\|y-c\|_{2})\|y-c\|_{2}+\frac{4}{3}Lh.
\end{align*}
Note that $\frac{4}{3}Lh\leq\frac{\tilde{r}}{3p}$. Hence, if $\|y-c\|_{2}\leq\left(1-\frac{1}{3p}\right)\tilde{r}$,
we know that $\|y'-c\|_{2}\leq\tilde{r}$. Otherwise if $\|y-c\|_{2}\geq\left(1-\frac{1}{3p}\right)\tilde{r}$,
we know that 
\[
\omega(\|y-c\|_{2})\|y-c\|_{2}\geq\frac{4L}{\tilde{r}^{p+1}}\left(1-\frac{1}{3p}\right)^{p + 1}\tilde{r}^{p+1}\geq\frac{4}{3}L
\]
which implies $\|y'-c\|_{2}\leq\|y-c\|_{2}$. Hence, in both cases,
we have $\|y'-c\|_{2}\leq\tilde{r}$. This completes the induction.

Finally, we need to bound the number of iterations before the algorithm
terminates. Let $\mathcal{L}(y):=g(y)+\Phi(\|y-c\|_{2})$ where $\Phi$
is defined in \eqref{eq:Lip_Phi}. By Lemma \ref{lem:convolution_err},
we have that 
\[
\nabla^{2}\mathcal{L}\preceq\Big(\frac{L}{r}+\frac{5L\sqrt{d}}{\tilde{r}}\Big)\cdot I_{d}\preceq\frac{6L\sqrt{d}}{\tilde{r}}\cdot I_{d} 
\]
Hence, by smoothness we have 
\[
\mathcal{L}(y')\leq\mathcal{L}(y)-h\left\langle \nabla L(y),\delta_{y}\right\rangle +3\frac{L}{\tilde{r}}\sqrt{d}\cdot h^{2}\|\delta_{y}\|^{2}.
\]
Note that $\delta_{y}=\nabla\mathcal{L}(y)+\eta$ for some vector
$\eta$ such that $\|\eta\|_{2}\leq L\cdot\frac{\epsilon}{6}$ by
the approximation guarantee. Therefore 
\begin{align*}
\mathcal{L}(y') & \leq\mathcal{L}(y)-h\|\nabla\mathcal{L}(y)\|^{2}+h\|\nabla\mathcal{L}(y)\|\|\eta\|+3\frac{L}{\tilde{r}}\sqrt{d}h^{2}(2\|\nabla\mathcal{L}(y)\|^{2}+2\|\eta\|^{2})\\
 & \leq\mathcal{L}(y)-\frac{7h}{8}\|\nabla\mathcal{L}(y)\|^{2}+h\|\nabla\mathcal{L}(y)\|\|\eta\|+\frac{h}{8}\|\eta\|^{2}\\
 & \leq\mathcal{L}(y)-\frac{7h}{8}\|\nabla\mathcal{L}(y)\|^{2}+\frac{h}{2}\|\nabla\mathcal{L}(y)\|^{2}+\frac{h}{2}\|\eta\|^{2}+\frac{h}{8}\|\eta\|^{2}\\
 & \leq\mathcal{L}(y)-\frac{7h}{8}\Big(L\cdot\frac{2\epsilon}{3}\Big)^{2}+\frac{5h}{8}\Big(L\cdot\frac{\epsilon}{6}\Big)^{2}\\
 & =\mathcal{L}(y)-\frac{h}{3}L^{2}\epsilon^{2}=\mathcal{L}(y)-\frac{\tilde{r}L\epsilon^{2}}{144p\sqrt{d}}
\end{align*}
where we used that $\|\nabla\mathcal{L}(y)\|\geq\|\delta_{y}\|-\|\eta\|\geq\frac{2\epsilon}{3}L$
from the stopping criteria. This shows that $\mathcal{L}$ decreased
by $\frac{\tilde{r}L\epsilon^{2}}{144p\sqrt{d}}$ every iteration.
Since $\mathcal{L}$ has Lipschitz constant $L+4L=5L$ on $\|y-c\|\leq\tilde{r}$,
\[
\max_{\|y-c\|\leq\tilde{r}}\mathcal{L}(y)-\min_{\|y-c\|\leq\tilde{r}}\mathcal{L}(y)\leq10L\tilde{r}.
\]
Therefore the number of step is at most $\mathcal{O}(\frac{p\sqrt{d}}{\epsilon^{2}})$
and we have 
\[
\left\Vert \nabla g(y)+\omega(\|y-c\|_{2})\cdot(y-c)\right\Vert _{2}\leq\frac{L\epsilon}{6}+\frac{5\epsilon}{6}L\leq L\cdot\epsilon
\]
as claimed. \end{proof}

The above theorem shows that we can implement (1) a noisy gradient
oracle with $\beta=\frac{L\cdot\epsilon}{6}$; and (2) an optimization
oracle with $\alpha=0$ and $\delta=L\cdot\epsilon$. Since by Theorem~\ref{thm:optimization_oracle}
we have $\|y_{k+1}-\tilde{x}_{k}\|\leq\tilde{r}$, i.e., the output
of the optimization oracle is bounded in a ball of radius $\tilde{r}$
from the center, therefore $v_{y_{k+1}}$ as the vector
field formed by sampling satisfies $\|v_{y_{k+1}}-\nabla g(y_{k+1})\|\leq\frac{\delta}{6}$,
justifying its validity as a noisy gradient oracle at $y_{k+1}$.

\subsection{Parallel complexity}
\label{sec:parallel_complexity}

Here we show how to put everything together to prove Theorem~\ref{thm:parallel_min_main}, our main highly-parallel optimization result.

\begin{proof}[Proof of Theorem~\ref{thm:parallel_min_main}]
Invoking the result of Section~\ref{sec:acceleration} and following the discussion in Section~\ref{sec:convolution}, with $r = \frac{\epsilon}{\sqrt{d}L}$, we have $\tilde{r} = \frac{r}{\sqrt{\log(\frac{60}{\epsilon'})}} = \frac{\epsilon}{L\sqrt{d\log(\frac{60}{\epsilon'})}}$ and since 
\[\omega(x) = \frac{4Lx^p}{\tilde{r}^{p+1}} = \frac{4L^{p+2}x^p[d\log(\frac{60}{\epsilon'})]^{\frac{p+1}{2}}}{\epsilon^{p+1}}\, ,\]
from Theorem~\ref{thm:main_acceleration} we have for $\frac{\gamma^2}{c} =\frac{p^2}{c} \leq \frac{1}{64}$, the convergence rate to an $\epsilon$-optimal point as 
\begin{align*}
f(y_{k})-f^{*} 
=\mathcal{O}\Big( \frac{\omega(\frac{\|x^*\|}{k^{3/2}})}{k^2}\|x^*\|^2\Big)
=\mathcal{O}\Big( \frac{ L^{p+2}\|x^*\|^p [d\log(\frac{1}{\epsilon'})]^{\frac{p+1}{2}}}{\epsilon^{p+1}\cdot k^2\cdot k^{\frac{3p}{2}}} \|x^*\|^2\Big)
\end{align*}
with $\mathcal{O}\Big(\frac{d\log d\log(\frac{1}{\epsilon'})+\log(\frac{1}{\rho})}{\epsilon'^{2}}\times \mathcal{K} \Big)$ (sub)gradient queries to $f$ in parallel in each round for $\epsilon' =  \mathcal{O}(\frac{\epsilon}{\|x^*\|\cdot L})$, as required by the accuracy for which the optimization oracle is implemented in Theorem~\ref{thm:main_acceleration} and the number of proximal oracle calls the line search procedure needs where 
\[
\mathcal{K}\defeq \Big(6+\log_2\Big[\frac{1500 \mu^2 R^2 c^2 [(1 + \alpha) c + 1]}{\epsilon}\omega(8c\mu R)\Big]\Big)^2
 = \mathcal{O}\Big(\log^2 \Big[\frac{L^{p+2}\|x^*\|^{p + 2}[d\log(\frac{1}{\epsilon'})]^{\frac{p+1}{2}}}{\epsilon^{p+2}}\Big]\Big) ~.
 \]
Setting the result to the desired accuracy $\epsilon$, we have that it suffices to pick $k = K$ for 
\begin{align*}
K &= \mathcal{O} \Big( \Big[ L^{p+2}\cdot  \|x^*\|^{p+2}\Big]^{\frac{2}{3p+4}}\cdot \Big[\frac{[d\log(\frac{\|x^*\|\cdot L}{\epsilon})]^{\frac{p+1}{2}}}{\epsilon^{p+2}}\Big]^{\frac{2}{3p+4}}\Big)\\
& =\mathcal{O} \Big(  \Big[L^{p+2}\cdot  \|x^*\|^{p+2}\Big]^{\frac{2}{3p+4}}\cdot \Big(\frac{d}{\epsilon^2}\Big)^{\frac{p+1}{3p+4}}\Big(\frac{1}{\epsilon}\Big)^{\frac{2}{3p+4}}\cdot \Big[\log\Big(\frac{\|x^*\|\cdot L}{\epsilon}\Big)\Big]^{\frac{p+1}{3p+4}}\Big)
\end{align*}
Picking $p$ such that $\log(\frac{d}{\epsilon^2}) = 3(3p+4)$, end up with
\[K =\mathcal{O}\Big( \Big(\frac{d}{\epsilon^2}\Big)^{\frac{1}{3}}\cdot\Big(\frac{1}{\epsilon}\Big)^{\frac{1}{\log(d/\epsilon^{2})}} \cdot \log^{\frac{1}{3}}\Big(\frac{1}{\epsilon}\Big) \cdot \Big(\log\Big(\frac{1}{\epsilon}\Big)\Big)^{\frac{1}{\log(d/\epsilon^{2})}}\Big)\]
which is $\tilde{\mathcal{O}}(d^{1/3}\epsilon^{-2/3})$, as claimed. Setting $\rho = \mathcal{O}(\frac{\nu}{K})$ for the algorithm to succeed with probability at least $1-\nu$, denote $\eta\defeq \log(\frac{d}{\epsilon^2})$ the number of parallel (sub)gradient queries is 
\begin{align*}
\mathcal{O}&\Big(\frac{d\log d\log(\frac{1}{\epsilon})+\log(d^{1/3}\epsilon^{-2/3}/\nu)}{\epsilon^{2}}\times \mathcal{K}\Big)\\
&=\mathcal{O}\Big(\frac{d\log d\log(\frac{1}{\epsilon})+\log(d^{1/3}\epsilon^{-2/3}/\nu)}{\epsilon^{2}}\times\log^2 \Big[\frac{[d\log(\frac{1}{\epsilon})]^{\frac{p+1}{2}}}{\epsilon^{p+2}}\Big]\Big)\\
&=\mathcal{O}\Big(\frac{d\log d\log(\frac{1}{\epsilon})+\log(d^{1/3}\epsilon^{-2/3}/\nu)}{\epsilon^{2}}\times\log^2 \Big[\frac{[d\log(\frac{1}{\epsilon})]^{\frac{1}{18}\eta-\frac{1}{6}}}{\epsilon^{\frac{1}{9}\eta+\frac{2}{3}}}\Big]\Big)
\end{align*}

With the choice of $p$, it suffices to pick $c$ large enough such that $\frac{81c}{64} \geq (\log(\frac{d}{\epsilon^2})-12)^2$ for the assumption to hold.
\end{proof}

\global\long\def\Rd{\mathbb{R}^{d}}

\section{Line search implementation}
\label{sec:implementation}

In this section, we assume access to an $(\alpha, \delta)$-approximate $\omega$-proximal step oracle $\oracleprox$ for a convex function $g$. The goal is to use $\oracleprox$ to find a point $y$ that satisfies Lemma~\ref{lem:framework_consistency}, as required by the algorithm framework at each iteration. The main result of this section is to prove the following theorem, originally stated in Appendix~\ref{sec:acceleration}. 

\lsrestate*

We assume $\delta\leq \frac{\epsilon}{\mu\cdot R\cdot 9c((1+\alpha)c+1)}$
to make sure
the oracle gives out information for different $x$ (and therefore we can achieve sufficiently small error). Furthermore, we assume $\delta \leq 8\mu R\cdot\omega(8\mu R)$ to ensure that if both $x$ and $y$ lie in a radius $\mu R$ ball then $\alpha\cdot\omega(\|y-x\|_{2})\cdot\|y-x\|_{2}$ is bounded by $2\mu R\cdot\omega(2\mu R)$. So if $\delta$ is much larger than this, the oracle essentially can always output the same $y$ regardless of $x$.

\subsection{Line search algorithm}

To simplify the notation, for the remainder of this section we make the following definitions. For all $\theta \in [0,1]$ we let $\widetilde{x}_{\theta}\defeq(1-\theta)x^{(1)}+\theta x^{(2)}$, $y_{\theta} \defeq \oracleprox(\widetilde{x}_{\theta})$, and $\lambda_{\theta} \defeq \frac{(1-\theta)^{2}A}{\theta}$. In this notation, our goal is to find $\theta \in [0, 1]$ such that
\begin{equation}
\frac{1}{2}\leq\zeta(\theta)\leq 1 \quad\text{where}\quad\zeta(\theta) \defeq \lambda_{\theta}\cdot\omega(\|y_{\theta}-\widetilde{x}_{\theta}\|_{2}) ~.\label{eq:zeta}
\end{equation}
Further, we let $\epsilon' \defeq \epsilon / ( 9c((1+\alpha)c+1) )$.

First, we show that $\zeta(0)=+\infty$ and $\zeta(1)=0$ or otherwise,
we find an approximate minimizer).
\begin{lem}
$\zeta(1)=0$ and either $\zeta(0)=+\infty$ or $g(x^{(1)})\leq g(x^{*})+\epsilon' \leq g(x^*) + \epsilon$.
\end{lem}

\begin{proof}
The claim $\zeta(1)=0$ follows immediately from definition. For the next claim, recall that by the definition of the $(\alpha,\delta)$ proximal oracle, we have that $y_0 = \oracleprox(x^{(1)})$ satisfies
\[
\norm{\nabla g(y_0)+\omega(\|y_0-x^{(1)}\|)\cdot(y_0-x^{(1)})} 
\leq\alpha\cdot\omega(\|y_0-x^{(1)}\|)\cdot\|y_0-x^{(1)}\|+\delta.
\]
If $\|y_0-x^{(1)}\|=0$,
we have $y_0=x^{(1)}$ and hence $\|\nabla g(x^{(1)})\|\leq\delta.$
By convexity of $g$, we have that from the assumption on diameter 
\[
g(x^{(1)})\leq g(x^{*})+\delta\|x^{(1)}-x^{*}\|_{2}\leq g(x^{*})+\mu\delta R\leq g(x^*)+\epsilon'.
\]
where we used $\delta \leq \frac{\epsilon'}{\mu\cdot R}$ at the end. Otherwise, we have $\|y_0-x^{(1)}\|>0$ therefore
$\zeta(0)=+\infty$ and $\zeta(1)=0$ from the definition.
\end{proof}

Therefore, to find $\theta$ such that $\zeta(\theta)=\frac{3}{4}$,
we simply perform binary search. In particular, in $\log_{2}(\frac{1}{\tau})$
iterations, we can find $0\leq\ell\leq u\leq1$ with $|\ell-u|\leq\tau$
such that $\zeta(\ell)-\frac{3}{4}$ and $\zeta(u)-\frac{3}{4}$ have
different signs. See Algorithm \ref{alg:binary} for the algorithm details. The key question is how small $\tau$
we need to make sure $\frac{1}{2}\leq\zeta(\frac{\ell+u}{2})\leq1$.

The difficultly here is that $\zeta$ may not be continuous. Therefore,
we cannot bound the Lipschitz constant of $\zeta$ directly. In contrast to previous papers \citep{BJLLS8}, our proof does not depend
on how we implement the proximal oracle $\oracleprox$ and do not directly assume how
$\oracleprox(x)$ changes with respect to $x$. In fact, the oracle
$\oracleprox$ we constructed in Section~\ref{sec:parallel_upper} may not even give
the same output for the same input. Therefore, it is difficult to
bound how far $\oracleprox(x)$ changes under the change of $\lambda$.
To avoid this problem, we first relate the noisy oracle $\oracleprox$
with the ideal oracle with $\alpha=\delta=0$. We note that the ideal
oracle is exactly performing a proximal step as follows:

\begin{lem}[Exact Proximal Map]
\label{lem:exact_prox}
Given $x$, let $y^*:= \mathcal{O}(x) :=\argmin_{y}\, G(y)$ where
\[
G(y) \defeq g(y)+W(\|y-x\|_{2})\quad\text{with}\quad W(s)\defeq\int_{0}^{s}\omega(u)\cdot u\,du
\]
then $\mathcal{O}$ is a $(0,0)$ proximal oracle for $g$. Further,
$\nabla^{2}G(y)\succeq\omega(\|y-x\|_{2})\cdot I$
for any $x$.
\end{lem}

\begin{proof}
From the optimality condition we have for $y^*=\mathcal{O}(x)$ 
\[
\nabla G(y^*)=\nabla g(y^*)+\omega(\|y^*-x\|_{2})\cdot(y^*-x)=0.
\]
which means $\mathcal{O}$ is a $(0,0)$ proximal oracle according to the definition.
Note that
\begin{align*}
\nabla^{2}G(y) & =\nabla^{2}g(y)+\omega(\|y-x\|_{2})I+\omega'(\|y-x\|_{2})\cdot\frac{(y-x)(y-x)^{\top}}{\|y-x\|_{2}}\\
 & \succeq\omega(\|y-x\|_{2})I
\end{align*}
where we used that $g$ is convex and $\omega$ is increasing.
Further, this
shows that $y^*$ is the unique minimizer of $G$.
\end{proof}

In Section \ref{subsec:exact_inexact_proximal}, we show that $\zeta$
is close to some continuous function $\zeta^{*}$ (except for some
special cases which we can handle separately).

\begin{figure}
\begin{algorithm2e}[H]\label{alg:binary}

\caption{Line Search Algorithm}

\SetKwRepeat{Do}{do}{while}

\SetAlgoLined

\textbf{Input}: $x^{(1)},x^{(2)}\in\Rd$ and $\frac{1}{2\omega(2\mu R)}\leq A\leq\frac{R^2}{\epsilon}$.\\
\textbf{Input:} $\epsilon' \defeq \frac{\epsilon}{9c((1+\alpha)c+1)} \in (0,1]$.\\
\textbf{Input}: an $(\alpha,\delta)$ proximal oracle $\oracleprox$
for a convex twice-differentiable function $g$.

\textbf{Assumption:} $\|x^{(1)}-x^*\|_{2}\leq \mu R$, $\|x^{(2)}-x^*\|_{2}\leq \mu R$, $\|x^{*}\|_{2}\leq R$ for some minimizer $x^{*}$ of $g$, which implies $\|x^{(1)}\|_2\leq 2\mu R$ and $\|x^{(2)}\|_2\leq 2\mu R$.

\textbf{Assumption: }$\delta \leq \min\{\frac{\epsilon'}{\mu\cdot R},8\mu R\cdot\omega(8\mu R)\}$. $0<\omega'(s)\leq\gamma\frac{\omega(s)}{s}$
for all $s>0$. $\frac{1-\sigma}{1-\alpha} = \frac{1}{2}$. $64(\alpha+\frac{1}{c})\gamma^2\leq1$ for some $c \geq 1$. 

Define $\widetilde{x}_{\theta}=(1-\theta)x^{(1)}+\theta x^{(2)}$, $y_{\theta}=\oracleprox(\widetilde{x}_{\theta})$
and $\zeta(\theta)$ according to \eqref{eq:zeta}.

Let $\tau=\min\Big\{\frac{1}{4},\frac{1}{2}\sqrt{\frac{1}{4}\frac{1}{A\cdot\omega(8c\mu R)}},\frac{A\delta}{64\mu R},\frac{c\delta}{360\mu\gamma R\cdot\omega(8c\mu R)},\frac{1}{200\left(1+A\cdot\omega(8c\mu R)+\frac{4\mu R}{A\delta}+\frac{\mu R}{\delta}\cdot\omega(8c\mu R)\right)}\Big\}$.

Set $\ell=0$, $u=1$.

\While{$u\geq\ell+\tau$}{

$m=\frac{\ell+u}{2}$.

\uIf{$\zeta(m)\geq\frac{3}{4}$}{

$\ell\leftarrow m$.

}\Else{

$u\leftarrow m$.

}

}

\uIf{$\omega(\|y_{\ell}-\widetilde{x}_{\ell}\|_{2})\cdot\|y_{\ell}-\widetilde{x}_{\ell}\|_{2}\leq c\cdot\delta$}{

\textbf{Return} $y_{\ell}$ as an approximate minimizer.

}\uElseIf{$\omega(\|y_{u}-\widetilde{x}_{u}\|_{2})\cdot\|y_{u}-\widetilde{x}_{u}\|_{2}\leq c\cdot\delta$}{

\textbf{Return} $y_{u}$ as an approximate minimizer.

}\Else{

\textbf{Return} $y_{\ell}$ as an approximate solution for the line
search.

}

\end{algorithm2e}
\end{figure}

\subsection{Line search regime: relation between exact and inexact proximal map}
\label{subsec:exact_inexact_proximal}

The goal of this section is to relate  
\[\zeta(\theta) \defeq \frac{(1-\theta)^{2}A}{\theta}\omega(\|y_{\theta}-\widetilde{x}_{\theta}\|_{2})\]
where $y_\theta = \oracleprox(\widetilde{x}_\theta)$ is output of an $(\alpha,\delta)$ proximal oracle to 
\[
\zeta^{*}(\theta) \defeq \frac{(1-\theta)^{2}A}{\theta}\omega(\|y_{\theta}^{*}-\widetilde{x}_{\theta}\|_{2})
\]
where $y_{\theta}^{*}=\argmin_{y}\, G_{\theta}(y)$ with
\begin{equation}
G_{\theta}(y) 
\defeq	
g(y)+W(\|y-\widetilde{x}_{\theta}\|_{2}) \, ,\label{eq:F_binary_search}
\end{equation}
the exact proximal map. In particular, we will show in Lemma \ref{lem:zeta_approximation} that
$\zeta(\theta)$ is an constant approximation of $\zeta^{*}(\theta)$.
Therefore, one can study the binary search of $\zeta$ via $\zeta^{*}$.

First we give a lemma that relates $\|y_{\theta}-\widetilde{x}_{\theta}\|_{2}$ and $\|y_{\theta}^{*}-\widetilde{x}_{\theta}\|_{2}$.
\begin{lem}
\label{lem:z_approximation} If $8(\alpha+\frac{1}{c})\gamma\leq1$ and
$\omega(\|y_{\theta}-\widetilde{x}_{\theta}\|_{2})\cdot\|y_{\theta}-\widetilde{x}_{\theta}\|_{2}\geq c\cdot \delta$,
then
\[
\left(1-8\left(\alpha+\frac{1}{c}\right)\gamma\right)\|y_{\theta}-\widetilde{x}_{\theta}\|_{2}\leq\|y_{\theta}^{*}-\widetilde{x}_{\theta}\|_{2}
\leq \left(1+8\left(\alpha+\frac{1}{c}\right)\gamma\right) \|y_{\theta}-\widetilde{x}_{\theta}\|_{2}.
\]
\end{lem}

\begin{proof}
We define $y_{\theta}^{(t)}=(1-t)y_{\theta}+ty_{\theta}^{*}$. Then,
we have that 
\begin{equation}
\nabla G_{\theta}(y_{\theta}^{*})-\nabla G_{\theta}(y_{\theta})=\int_{0}^{1}\nabla^{2}G_{\theta}(y_{\theta}^{(t)})\cdot(y_{\theta}^*-y_{\theta})dt.\label{eq:zeta_apx_1}
\end{equation}
Lemma \ref{lem:exact_prox} shows that
\begin{equation}
\nabla^{2}G_{\theta}(y_{\theta}^{(t)})\succeq\omega(\|y_{\theta}^{(t)}-\widetilde{x}_{\theta}\|_{2})\cdot I.\label{eq:zeta_apx_2}
\end{equation}

To lower bound $\|y_{\theta}^{(t)}-\widetilde{x}_{\theta}\|_{2}$,
we split the proof into two cases:

Case 1: $\|y_{\theta}-y_{\theta}^{*}\|_{2}\geq4\|y_{\theta}-\widetilde{x}_{\theta}\|_{2}$.
Since $y_{\theta}^{(t)}=(1-t)y_{\theta}+ty_{\theta}^{*}$, then for
$t\geq\frac{1}{2}$,
\begin{align*}
\|y_{\theta}^{(t)}-\widetilde{x}_{\theta}\|_{2}  =\|y_{\theta}-\widetilde{x}_{\theta}+t(y_{\theta}^{*}-y_{\theta})\|_{2}
 \geq t\|y_{\theta}^{*}-y_{\theta}\|_{2}-\|y_{\theta}-\widetilde{x}_{\theta}\|_{2}
  \geq\|y_{\theta}-\widetilde{x}_{\theta}\|_{2}.
\end{align*}
Since $\omega$ is increasing  $\omega(\|y_{\theta}^{(t)}-\widetilde{x}_{\theta}\|_{2})\geq\omega(\|y_{\theta}-\widetilde{x}_{\theta}\|_{2})$.
Combining with \eqref{eq:zeta_apx_1} and \eqref{eq:zeta_apx_2} this yields 
\begin{align*}
\|\nabla G_{\theta}(y_{\theta})-\nabla G_{\theta}(y_{\theta}^{*})\|_{2}  \geq\int_{1/2}^{1}\omega(\|y_{\theta}-\widetilde{x}_{\theta}\|_{2})dt\cdot\|y_{\theta}-y_{\theta}^{*}\|_{2}
  =\frac{1}{2}\omega(\|y_{\theta}-\widetilde{x}_{\theta}\|_{2})\cdot\|y_{\theta}-y_{\theta}^{*}\|_{2}.
\end{align*}

Case 2: $\|y_{\theta}-y_{\theta}^{*}\|_{2}\leq4\|y_{\theta}-\widetilde{x}_{\theta}\|_{2}$.
Since $y_{\theta}^{(t)}=(1-t)y_{\theta}+ty_{\theta}^{*}$, we have
\[
\|y_{\theta}^{(t)}-\widetilde{x}_{\theta}\|_{2}\geq\|y_{\theta}-\widetilde{x}_{\theta}\|_{2}-t\|y_{\theta}^{*}-y_{\theta}\|_{2}\geq(1-4t)\|y_{\theta}-\widetilde{x}_{\theta}\|_{2}.
\]
Using this and $\omega(\eta\cdot\beta)\leq\eta^{\gamma}\omega(\beta)$
(which is implied by $\omega'(s)\leq\gamma\frac{\omega(s)}{s}$ from Gr\"onwall's inequality),
for $0\leq t\leq\frac{1}{4}$, we have that
\[
\omega(\|y_{\theta}^{(t)}-\widetilde{x}_{\theta}\|_{2})\geq(1-4t)^{\gamma}\omega(\|y_{\theta}-\widetilde{x}_{\theta}\|_{2}).
\]
Together with \eqref{eq:zeta_apx_1} and \eqref{eq:zeta_apx_2}, we
have that
\begin{align}
\|\nabla G_{\theta}(y_{\theta})-\nabla G_{\theta}(y_{\theta}^{*})\|_{2} & \geq\int_{0}^{1/4}(1-4t)^{\gamma}dt\cdot\omega(\|y_{\theta}-\widetilde{x}_{\theta}\|_{2})\cdot\|y_{\theta}-y_{\theta}^{*}\|_{2}\nonumber \\
 & =\frac{1}{4(\gamma+1)}\cdot\omega(\|y_{\theta}-\widetilde{x}_{\theta}\|_{2})\cdot\|y_{\theta}-y_{\theta}^{*}\|_{2}.\label{eq:zeta_apx_3}
\end{align}
In both cases, we have \eqref{eq:zeta_apx_3} as $\gamma \geq 1$. 

On the other hand, the assumption on $y_{\theta}$ shows that 
\begin{align}
\|\nabla G_{\theta}(y_{\theta})-\nabla G_{\theta}(y_{\theta}^{*})\|_{2}=\|\nabla G_{\theta}(y_{\theta})\|_{2} & \leq\alpha\cdot\omega(\|y_{\theta}-\widetilde{x}_{\theta}\|_{2})\cdot\|y_{\theta}-\widetilde{x}_{\theta}\|_{2}+\delta\nonumber \\
 & \leq \left(\alpha+\frac{1}{c}\right)\cdot\omega(\|y_{\theta}-\widetilde{x}_{\theta}\|_{2})\cdot\|y_{\theta}-\widetilde{x}_{\theta}\|_{2}\label{eq:zeta_eqv_2}
\end{align}
where we used the assumption on $\omega(\|y_{\theta}-\widetilde{x}_{\theta}\|_{2})\cdot\|y_{\theta}-\widetilde{x}_{\theta}\|_{2}$.

Combining \eqref{eq:zeta_apx_3} and \eqref{eq:zeta_eqv_2}, we have
that
\[
\|y_{\theta}-y_{\theta}^{*}\|_{2}\leq4 \left(\alpha+\frac{1}{c} \right)(\gamma+1)\cdot\|y_{\theta}-\widetilde{x}_{\theta}\|_{2}
\leq
8\left(\alpha+\frac{1}{c}\right)\gamma\cdot\|y_{\theta}-\widetilde{x}_{\theta}\|_{2}
\]
where we used that $\gamma\geq1$. The claim now follows from triangle inequality.
\end{proof}

Since $\zeta = \lambda_\theta \cdot \omega(\|y_{\theta}-\widetilde{x}_{\theta}\|_{2})$ and $\lambda_\theta$ is just a closed form function $\lambda$ 
we have the following main result of this section:
\begin{lem}
\label{lem:zeta_approximation}If $64(\alpha+\frac{1}{c})\gamma^2\leq1$ and $\omega(\|y_{\theta}-\widetilde{x}_{\theta}\|_{2})\cdot\|y_{\theta}-\widetilde{x}_{\theta}\|_{2}\geq c\cdot \delta$,
then $\frac{7}{8}\zeta(\theta)\leq\zeta^{*}(\theta)\leq\frac{5}{4}\zeta(\theta)$.
\end{lem} 

\begin{proof}
Lemma \ref{lem:z_approximation} shows that 
\[
\left( 1-8 \left(\alpha+\frac{1}{c} \right)\gamma \right)\|y_{\theta}-\widetilde{x}_{\theta}\|_{2}\leq\|y_{\theta}^{*}-\widetilde{x}_{\theta}\|_{2}\leq
\left(1+8\left(\alpha+\frac{1}{c}\right)\gamma \right)\|y_{\theta}-\widetilde{x}_{\theta}\|_{2}.
\]
Using $\omega$ is non-decreasing and $\omega(\eta\cdot\beta)\leq\eta^{\gamma}\omega(\beta)$,
we have
\[
\left(1-8\left(\alpha+\frac{1}{c} \right)\gamma\right)^{\gamma}\omega(\|y_{\theta}-\widetilde{x}_{\theta}\|_{2})\leq\omega(\|y_{\theta}^*-\widetilde{x}_{\theta}\|_{2})
\leq
\left(1+8\left(\alpha+\frac{1}{c} \right)\gamma\right)^{\gamma}\omega(\|y_{\theta}-\widetilde{x}_{\theta}\|_{2}).
\]
The result now follows from the assumption $64(\alpha+\frac{1}{c})\gamma^2\leq1$. 
\end{proof}

\subsection{Approximate minimization regime\texorpdfstring{: when $y_{\theta}$ is close to $\widetilde{x}_{\theta}$}{}}

In Section \ref{subsec:exact_inexact_proximal}, we show that if $\|y_{\theta}-\widetilde{x}_{\theta}\|_{2}$
is large, $\zeta$ approximates $\zeta^{*}$ up to constant factor. In this section, we
handle the other case. We show that if $\|y_{\theta}-\widetilde{x}_{\theta}\|_{2}$
is small, then we can find a $y$ with small function value $g(y)$.
First, we show that $\|y_{\theta}-\widetilde{x}_{\theta}\|_{2}$ cannot
be too large.
\begin{lem}
\label{lem:upperbound_z} If $16(\alpha+\frac{1}{c})\gamma\leq1$ then $
\|y_{\theta}-\widetilde{x}_{\theta}\|_{2}\leq 8c\mu R
$
for all $\theta\in[0,1]$.
\end{lem}

\begin{proof}
Case 1: $\omega(\|y_{\theta}-\widetilde{x}_{\theta}\|_{2})\cdot\|y_{\theta}-\widetilde{x}_{\theta}\|_{2}\geq c\cdot \delta$.
Using this and $16(\alpha+\frac{1}{c})\gamma\leq1$, Lemma \ref{lem:z_approximation}
shows that
\begin{equation}
\|y_{\theta}-\widetilde{x}_{\theta}\|_{2}\leq 2\|y_{\theta}^{*}-\widetilde{x}_{\theta}\|_{2}.\label{eq:upperbound_z_3}
\end{equation}
To upper bound $\|y_{\theta}^{*}-\widetilde{x}_{\theta}\|_{2}$, we
use the fact that $y_{\theta}^{*}$ is the minimizer of $G_{\theta}$
and get
\[
g(x^{*})+W(\|x^{*}-\widetilde{x}_{\theta}\|_{2})=G_{\theta}(x^{*})\geq G_{\theta}(y_{\theta}^{*})\geq g(x^{*})+W(\|y_{\theta}^{*}-\widetilde{x}_{\theta}\|_{2}).
\]
Since $W$ is increasing, we have $\|y_{\theta}^{*}-\widetilde{x}_{\theta}\|_{2}\leq\|x^{*}-\widetilde{x}_{\theta}\|_{2}\leq  \mu R$
where we used the assumption on $x^{(1)}$ and $x^{(2)}$. 
Putting these into \eqref{eq:upperbound_z_3} yields the result.

Case 2: $\omega(\|y_{\theta}-\widetilde{x}_{\theta}\|_{2})\cdot\|y_{\theta}-\widetilde{x}_{\theta}\|_{2}\leq c\cdot \delta$.
Since $\delta\leq 8\mu R\cdot\omega(8\mu R)$ and $\omega$ is increasing,
we have 
\[
\|y_{\theta}-\widetilde{x}_{\theta}\|_{2}\leq8c\mu R.
\]
Therefore in both cases we have $\|y_{\theta}-\widetilde{x}_{\theta}\|_{2}\leq8c\mu R$ as $c \geq 1$.
\end{proof}

Now, we show that small $\|y_{\theta}-\widetilde{x}_{\theta}\|_{2}$
implies small $g(y_{\theta})$.
\begin{lem}
\label{lem:z_lower} If $\omega(\|y_{\theta}-\widetilde{x}_{\theta}\|_{2})\cdot\|y_{\theta}-\widetilde{x}_{\theta}\|_{2}\leq c\cdot \delta$, we have that
$
g(y_{\theta})\leq g(x^{*})+\epsilon
$.
\end{lem}

\begin{proof}
By the definition of $y_{\theta}$ and the assumption, we have
\begin{align}
\|\nabla g(y_{\theta})\|_{2} & \leq(1+\alpha)\omega(\|y_{\theta}-\widetilde{x}_{\theta}\|_{2})\cdot\|y_{\theta}-\widetilde{x}_{\theta}\|_{2}+\delta\leq ((1+\alpha)c+1)\delta\label{eq:z_lower_v}
\end{align}
Hence, convexity of $g$ shows that
\begin{align*}
g(y_{\theta})-g(x^{*}) & \leq\left\langle \nabla g(y_{\theta}),y_{\theta}-x^{*}\right\rangle \leq ((1+\alpha)c+1)\delta\|y_{\theta}-x^{*}\|_{2}.
\end{align*}
To bound $\|y_{\theta}-x^{*}\|_{2}$, we note that
\[
\|y_{\theta}-x^{*}\|_{2}\leq \|\widetilde{x}_{\theta}-x^*\|_{2} +\|y_{\theta}-\widetilde{x}_{\theta}\|_{2}\leq\mu R + 8c\mu R\leq9c\mu R
\]
where we used Lemma \ref{lem:upperbound_z} and the assumption on diameter.
Hence, convexity of $g$ shows that
\begin{align*}
g(y_{\theta})-g(x^{*}) & \leq\left\langle \nabla g(y_{\theta}),y_{\theta}-x^{*}\right\rangle \leq ((1+\alpha)c+1)\delta\cdot9c\mu R \leq 9c((1+\alpha)c+1)\epsilon' \, .
\end{align*}
where we used $\delta\leq \frac{\epsilon'}{\mu\cdot R}$.
\end{proof}

\subsection{Bounding Lipschitz constant\texorpdfstring{ of $\zeta^{*}(\theta)$}{}}

To derive the stopping criteria $\tau$ (and therefore the iteration complexity), we need to bound the Lipschitz constant of $\zeta^*(\theta)$. We first give
an upper bound on $\|\frac{d}{d\theta}(y_{\theta}^{*}-\widetilde{x}_{\theta})\|$. 
\begin{lem}
\label{lem:z_speed}We have:
\[
\left\Vert \frac{d}{d\theta}(y_{\theta}^{*}-\widetilde{x}_{\theta})\right\Vert \leq12\mu\gamma R.
\]
\end{lem}

\begin{proof}
To compute the derivative of $y_{\theta}$, we note by optimality
condition that 
\[
\nabla G_{\theta}(y_{\theta}^{*})=0.
\]
Taking derivatives with respect to $\theta$ on both sides gives 
\[
\frac{d}{d\theta}\nabla G_{\theta}(y_{\theta}^{*})+\nabla^{2}G_{\theta}(y_{\theta}^{*})\cdot\frac{d}{d\theta}y_{\theta}^{*}=0.
\]
Hence, we have 
\begin{equation}
\frac{d}{d\theta}y_{\theta}^{*}=-\left(\nabla^{2}G_{\theta}(y_{\theta}^{*})\right)^{-1}\left(
\left( \frac{d}{d\theta}\nabla G_{\theta} \right) 
(y_{\theta}^{*})\right).
\label{eq:dz_dtheta}
\end{equation}
To bound $\frac{d}{d\theta}y_{\theta}^{*}$, we first compute $\frac{d}{d\theta}\nabla G_{\theta}(y)$
and $\nabla^{2}G_{\theta}(y)$. For $\frac{d}{d\theta}\nabla G_{\theta}(y)$,
we have
\begin{align*}
\frac{d}{d\theta}\nabla G_{\theta}(y) & =\frac{d}{d\theta}\left[ \omega(\|y-\widetilde{x}_{\theta}\|_{2})\cdot(y-\widetilde{x}_{\theta})\right]\\
 & =-\omega'(\|y-\widetilde{x}_{\theta}\|_{2})\cdot\frac{(y-\widetilde{x}_{\theta})(y-\widetilde{x}_{\theta})^{\top}}{\|y-\widetilde{x}_{\theta}\|_{2}}(x^{(2)}-x^{(1)})-\omega(\|y-\widetilde{x}_{\theta}\|_{2})\cdot(x^{(2)}-x^{(1)}).
\end{align*}
For $\nabla^{2}G_{\theta}(y)$, Lemma \ref{lem:exact_prox} shows
that 
\[
\nabla^{2}G_{\theta}(y)\succeq\omega(\|y-\widetilde{x}_{\theta}\|_{2})\cdot I.
\]

Now, \eqref{eq:dz_dtheta} shows 
\begin{align*}
\left\|\frac{d}{d\theta}y_{\theta}^{*}\right\| & \leq\left[\frac{\omega'(\|y_{\theta}^{*}-\widetilde{x}_{\theta}\|_{2})}{\omega(\|y_{\theta}-\widetilde{x}_{\theta}\|_{2})}\cdot\left|(y_{\theta}^{*}-\widetilde{x}_{\theta})^{\top}(x^{(2)}-x^{(1)})\right|+\|x^{(2)}-x^{(1)}\|_{2}\right]\\
 & \leq\frac{\omega'(\|y_{\theta}^{*}-\widetilde{x}_{\theta}\|_{2})}{\omega(\|y_{\theta}^{*}-\widetilde{x}_{\theta}\|_{2})}\cdot\|y_{\theta}^{*}-\widetilde{x}_{\theta}\|\cdot\|x^{(2)}-x^{(1)}\|+\|x^{(2)}-x^{(1)}\|_{2}\\
 & \leq(1+\gamma)\cdot\|x^{(2)}-x^{(1)}\|_{2}
\end{align*}
where we used that $\omega'(s)\leq\gamma\cdot\frac{\omega(s)}{s}$
at the end. Hence, we have
\[
\left\|\frac{d}{d\theta}(y_{\theta}^{*}-\widetilde{x}_{\theta})\right\|
\leq
\left\|\frac{d}{d\theta}y^*_{\theta}\right\| + \|x^{(2)}-x^{(1)}\|\leq(2+\gamma)\|x^{(2)}-x^{(1)}\|_{2}.
\]
The result follows from $\gamma\geq1$ and $\|x^{(2)}-x^{(1)}\|_{2}\leq4\mu R$.
\end{proof}

 We now give a bound on the Lipschitz constant $\zeta^{*}(\theta)$.
\begin{lem}
\label{lem:lipschitz}We have
\[
\left|\frac{d}{d\theta}\log\zeta^{*}(\theta)\right|\leq\frac{2}{1-\theta}+\frac{1}{\theta}+\frac{12\mu\gamma^2 R}{\|y_{\theta}^{*}-\widetilde{x}_{\theta}\|_{2}}.
\]
\end{lem}

\begin{proof}
Note that 
\[
\frac{d}{d\theta}\log\zeta^{*}(\theta)=-\frac{2}{1-\theta}-\frac{1}{\theta}+\frac{\omega'(\|y_{\theta}^{*}-\widetilde{x}_{\theta}\|_{2})}{\omega(\|y_{\theta}^{*}-\widetilde{x}_{\theta}\|_{2})}\frac{(y_{\theta}^{*}-\widetilde{x}_{\theta})^{\top}\frac{d}{d\theta}(y_{\theta}^{*}-\widetilde{x}_{\theta})}{\|y_{\theta}^{*}-\widetilde{x}_{\theta}\|_{2}}.
\]
Using $\omega'(s)\leq\gamma\cdot\frac{\omega(s)}{s}$, we have
\begin{align*}
\left|\frac{d}{d\theta}\log\zeta^{*}(\theta)\right| & \leq\frac{2}{1-\theta}+\frac{1}{\theta}+\gamma\frac{\left|(y_{\theta}^{*}-\widetilde{x}_{\theta})^{\top}\frac{d}{d\theta}(y_{\theta}^{*}-\widetilde{x}_{\theta})\right|}{\|y_{\theta}^{*}-\widetilde{x}_{\theta}\|_{2}^{2}}\\
 & \leq\frac{2}{1-\theta}+\frac{1}{\theta}+\gamma\frac{\|\frac{d}{d\theta}(y_{\theta}^{*}-\widetilde{x}_{\theta})\|_{2}}{\|y_{\theta}^{*}-\widetilde{x}_{\theta}\|_{2}}
 \leq\frac{2}{1-\theta}+\frac{1}{\theta}+\frac{12\mu\gamma^2 R}{\|y_{\theta}^{*}-\widetilde{x}_{\theta}\|_{2}}
\end{align*}
from Lemma \ref{lem:z_speed}.
\end{proof}

Since the Lipschitz constant of $\zeta^{*}$ depends on the term $\frac{1}{1-\theta}$
and $\frac{1}{\theta}$, we need to show that $\theta$ cannot be
too close to $0$ and $1$. First, we give an upper bound $\theta$.
\begin{lem}[Upper bound on $\theta$]
\label{lem:upper_bound_theta}Assume that $16(\alpha+\frac{1}{c})\gamma\leq1$.
For any $\theta \in [0,1]$ with $\frac{1}{2}\leq\zeta(\theta)$, we have
\[
\theta
\leq
\max
\left\{
\frac{1}{2}
~,~
1-\sqrt{\frac{1}{4 A\cdot\omega(8c\mu R)}}
\right\}
\]
In particular, we have $u\leq\max\{\frac{3}{4},1-\frac{1}{2}\sqrt{\frac{1}{4A\cdot\omega(8c\mu R)}}\}$ when Algorithm~\ref{alg:binary} terminates.
\end{lem}

\begin{proof}
Suppose that $\theta\geq\frac{1}{2}$, then we have
\[
\frac{1}{2}
\leq \zeta(\theta)
= \frac{(1-\theta)^{2}A}{\theta} \cdot \omega(\|y_{\theta}-\widetilde{x}_{\theta}\|_{2})
\leq 2(1-\theta)^{2}A \cdot \omega(\|y_{\theta}-\widetilde{x}_{\theta}\|_{2}).
\]
The bound on $\theta$ now follows from Lemma \ref{lem:upperbound_z}.
Since we stop the binary search when $|u-\ell|$ less than
$\frac{1}{2}\min \{\frac{1}{2},\sqrt{ \frac{1}{4A\cdot\omega(8c\mu R)}}\}$,
we have the upper bound on $u$.
\end{proof}
Next, we give a lower bound on $\theta$.

\begin{lem}[Lower bound on $\theta$]
\label{lem:lower_bound_theta}
Assume $16(\alpha+\frac{1}{c})\gamma\leq1$. For any $\theta\in[0,1]$ with $\zeta(\theta)\leq 1$
and $\omega(\|y_{\theta}-\widetilde{x}_{\theta}\|_{2})\cdot\|y_{\theta}-\widetilde{x}_{\theta}\|_{2}\geq c \cdot \delta$,
we have 
\[
\theta
\geq
\min \left\{
\frac{1}{2} , \frac{A\delta}{32\mu R}
\right\} ~.
\]
In particular, we have $\ell \geq \min \{\frac{1}{4} , \frac{A\delta}{64\mu R} \}$
or $\omega(\|y_{\theta}-\widetilde{x}_{\theta}\|_{2})\cdot\|y_{\theta}-\widetilde{x}_{\theta}\|_{2}\leq c\cdot \delta$ when Algorithm~\ref{alg:binary} terminates.
\end{lem}

\begin{proof}
Suppose that $\theta\leq\frac{1}{2}$, then we have from the assumption
\begin{align*}
1\geq\zeta(\theta) 
	& =\frac{(1-\theta)^{2}A}{\theta} \cdot \omega(\|y_{\theta}-\widetilde{x}_{\theta}\|_{2})
  	\geq\frac{1}{4}\cdot\frac{A}{\theta} \cdot \omega(\|y_{\theta}-\widetilde{x}_{\theta}\|_{2})\\
 	& \geq\frac{1}{4}\cdot\frac{A}{\theta}\frac{c\delta}{\|y_{\theta}-\widetilde{x}_{\theta}\|_{2}}
 	\geq\frac{1}{4}\cdot\frac{A}{\theta}\frac{c\delta}{8c\mu R}
\end{align*}
where we used Lemma \ref{lem:upperbound_z}. This gives the lower bound on $\theta$.
Since we stop the binary search when $|u-\ell|$ less than
$\frac{1}{2}\min \{  \frac{1}{2},\frac{A\delta}{32\mu R} \}$,
we have the lower bound on $\ell$.
\end{proof}

Now, we are ready to show the correctness of Algorithm \ref{alg:binary} with the assumed $\tau$.
\begin{theorem}[Correctness of Algorithm] \label{thm:correctness}
Assume $64(\alpha+\frac{1}{c})\gamma^2\leq1$. Algorithm \ref{alg:binary} outputs
either $y$ such that 
$g(y)\leq g^*+\epsilon$
or $y=y_{\theta}$ such that $\zeta(\theta) \in [1/2, 1]$
with
\[
\left\Vert \nabla g(y_{\theta})+\omega(\|y_{\theta}-\widetilde{x}_{\theta}\|_{2})\cdot(y_{\theta}-\widetilde{x}_{\theta})\right\Vert \leq\alpha\cdot\omega(\|y_{\theta}-\widetilde{x}_{\theta}\|_{2})\cdot\|y_{\theta}-\widetilde{x}_{\theta}\|_{2}+\delta
\]
where $\delta \leq \frac{1}{c}\omega(\|y_{\theta}-\widetilde{x}_{\theta}\|_{2})\cdot\|y_{\theta}-\widetilde{x}_{\theta}\|_{2}$.
\end{theorem}

\begin{proof}
For the case $\omega(\|y_{\ell}-\widetilde{x}_{\ell}\|_{2})\cdot\|y_{\ell}-\widetilde{x}_{\ell}\|_{2}\leq c\cdot \delta$
and $\omega(\|y_{u}-\widetilde{x}_{u}\|_{2})\cdot\|y_{u}-\widetilde{x}_{u}\|_{2}\leq c\cdot \delta$,
Lemma \ref{lem:z_lower} shows that $g(y)\leq g^*+\epsilon$.

Otherwise, Lemma \ref{lem:upper_bound_theta} and Lemma \ref{lem:lower_bound_theta}
show that 
\begin{equation}
\ell\geq\min\left\{\frac{1}{4},\frac{A\delta}{64\mu R}\right\}\label{eq:ell_lower}
\end{equation}
and 
\begin{equation}
u\leq\max\left\{\frac{3}{4},1-\frac{1}{2}\sqrt{\frac{1}{4}\frac{1}{A\cdot\omega(8c\mu R)}}\right\}\, . \label{eq:u_upper}
\end{equation}
Therefore, together with Lemma \ref{lem:lipschitz} we have
\begin{align}
\left|\frac{d}{d\theta}\log\zeta^{*}(\theta)\right| & \leq\frac{2}{1-\theta}+\frac{1}{\theta}+\frac{12\mu\gamma^2 R}{\|y_{\theta}^{*}-\widetilde{x}_{\theta}\|_{2}}\nonumber \\
 & \leq12+4\sqrt{4A\cdot\omega(8c\mu R)}+\frac{64\mu R}{A\delta}+\frac{12\mu\gamma^2 R}{\|y_{\theta}^{*}-\widetilde{x}_{\theta}\|_{2}}\label{eq:zetaLip}
\end{align}
for all $\ell\leq\theta\leq u$.
To bound the term $\|y_{\theta}^{*}-\widetilde{x}_{\theta}\|_{2}$, note from Lemma \ref{lem:zeta_approximation} we have
\begin{equation}
\frac{8}{7}\|y_{u}^{*}-\widetilde{x}_{u}\|_{2} \geq \|y_{u}-\widetilde{x}_{u}\|_{2}\geq\frac{c\delta}{\omega(\|y_{u}-\widetilde{x}_{u}\|_{2})}.\label{eq:zu_lower}
\end{equation}
Using $\frac{3}{4}\geq\zeta(u)$ (due to binary search),
we have
\[
\frac{3}{4}\geq\zeta(u)=\frac{(1-u)^{2}A}{u}\omega(\|y_{u}-\widetilde{x}_{u}\|_{2})\geq(1-u)^{2}A\omega(\|y_{u}-\widetilde{x}_{u}\|_{2}).
\]
Putting it into \eqref{eq:zu_lower} gives
\[
\|y_{u}^{*}-\widetilde{x}_{u}\|_{2}\geq \frac{28c\delta (1-u)^2A}{24} \geq \frac{7c\delta A}{6}\frac{1}{16 A\cdot \omega(8c\mu R)}\geq \frac{c\delta}{15 \cdot \omega(8c\mu R)}
\]
where we used \eqref{eq:u_upper} for the last inequality. Lemma \ref{lem:z_speed}
shows that 
\[
\left\Vert \frac{d}{d\theta}(y_{\theta}^{*}-\widetilde{x}_{\theta})\right\Vert \leq12 \mu\gamma R.
\]
Since we have from the stopping criteria $\tau = |u-\ell|\leq\frac{c\delta}{360\mu\gamma R\cdot\omega(8c\mu R)}$,
for all $\ell\leq\theta\leq u$, this gives
\[
\|y_{\theta}^{*}-\widetilde{x}_{\theta}\|_{2}\geq\|y_{u}^{*}-\widetilde{x}_{u}\|_{2}-12\mu\gamma R\cdot\tau\geq\frac{c\delta}{30\omega(8c\mu R)}.
\]

Put together with \eqref{eq:zetaLip} we have
\begin{align*}
\left|\frac{d}{d\theta}\log\zeta^{*}(\theta)\right| & \leq12+8\sqrt{A\cdot\omega(8c\mu R)}+\frac{64\mu R}{A\delta}+\frac{360\mu\gamma^{2}R\omega(8c\mu R)}{c\delta}\\
 & \leq20+20A\cdot\omega(8c\mu R)+\frac{64\mu R}{A\delta}+\frac{6\mu R\omega(8c\mu R)}{\delta}\\
 & \leq20\left(1+A\cdot\omega(8c\mu R)+\frac{4\mu R}{A\delta}+\frac{\mu R}{\delta}\cdot\omega(8c\mu R)\right)
\end{align*}
where we used $64(\alpha+\frac{1}{c})\gamma^2\leq1$ and $\alpha\leq1$. Due to
the choice of $\tau \leq \frac{1}{200\left(1+A\cdot\omega(8c\mu R)+\frac{4\mu R}{A\delta}+\frac{\mu R}{\delta}\cdot\omega(8c\mu R)\right)}$, this shows that $\zeta^{*}(\ell)\leq e^{\frac{1}{10}}\zeta^{*}(u)$.
Now, using Lemma \ref{lem:zeta_approximation}, we have
\[
\zeta(\ell)\leq \frac{8}{7}\zeta^{*}(\ell)\leq \frac{8}{7}e^{\frac{1}{10}}\zeta^{*}(u)\leq \frac{8}{7}e^{\frac{1}{10}} \frac{5}{4}\cdot \zeta(u)\leq \frac{8}{7}e^{\frac{1}{10}}\frac{5}{4}\frac{3}{4}\leq 1.
\]
Moreover, by the definition of binary search, we know $\zeta(\ell)\geq\frac{3}{4}$.
This completes the proof that we have found a point satisfying $\frac{1}{2}\leq\zeta(\theta)\leq 1.$
\end{proof}

\subsection{Bounding the number of steps}

To bound the number of steps, we need to have a lower and upper bound
on $A$. We note that when we apply the line search procedure, we have $A=A_{k}$ at iteration $k$.
Furthermore, we assume $k\geq1$ because no line search is needed for $k=0$. Under the assumption, we have $\frac{1}{2\omega(2\mu R)}\leq A\leq\frac{R^2}{\epsilon}$. Below we give the proof of the main theorem for the line search implementation. 

\begin{proof}[Proof of Theorem~\ref{thm:main_line_search}]
Recall from the algorithm description, we set
\begin{align*}
\frac{1}{\tau} & \leq4+2\sqrt{4A\cdot\omega(8c\mu R)}+\frac{64\mu  R}{A\delta}+\frac{360\mu\gamma R\cdot\omega(8c\mu R)}{c\delta}\\
&\quad \quad +200\left(1+A\cdot\omega(8c\mu R)+\frac{4\mu R}{A\delta}+\frac{\mu R}{\delta}\cdot\omega(8c\mu R)\right)\\
 & \leq300\left(1+A\cdot\omega(8c\mu R)+\frac{4\mu R}{A\delta}+\frac{\mu R}{\delta}\cdot\omega(8c\mu R)\right)
\end{align*}
where we used $16(\alpha+\frac{1}{c})\gamma\leq1$. Now using $\frac{1}{2\omega(2\mu R)}\leq A\leq\frac{R^2}{\epsilon}$ from the assumption we get
\begin{align*}
\frac{1}{\tau}
\leq 300\left(1+\Big(\frac{R^2}{\epsilon}+\frac{9\mu R}{\delta}\Big)\cdot\omega(8c\mu R)\right)
\leq 40\Big[\frac{160\mu R c}{\delta}+\frac{9R^2}{\epsilon}\Big]\cdot \omega(8c\mu R) 
\end{align*}
where we used $\delta\leq8\mu R\cdot\omega(8\mu R)$ at the end. Putting together with Theorem~\ref{thm:correctness} yields the result.
\end{proof}

\end{document}